\documentclass[final,hidelinks,onefignum,onetabnum]{siamart220329}
\usepackage{amsmath}
\usepackage{color,graphicx,subfigure}
\usepackage{amssymb}
\DeclareMathOperator{\diag}{diag}
\headers{Higher order derivatives of matrix functions}{Emanuel H. Rubensson}
\title{Higher order derivatives of matrix functions%
  \thanks{\textbf{Funding:} This work was supported by the Swedish strategic research programme eSSENCE.}}
\author{Emanuel H. Rubensson\thanks{Division of Scientific Computing, Department of Information Technology, Uppsala University, Box 337, SE-751 05 Uppsala, Sweden (\email{emanuel.rubensson@it.uu.se})}}

\ifpdf
\hypersetup{
  pdftitle={Higher order derivatives of matrix functions},
  pdfauthor={Emanuel H. Rubensson}
}
\fi

\begin{document}
\maketitle
\begin{abstract}
  We present theory for general partial derivatives of matrix
  functions on the form $f(A(x))$ where $A(x)$ is a matrix path of
  several variables ($x=(x_1,\dots,x_j)$). Building on results by
  Mathias [SIAM J. Matrix Anal. Appl., 17 (1996), pp. 610-620] for the first order derivative, we develop a block upper
  triangular form for higher order partial derivatives.
  This block
  form is used to derive conditions for existence and a generalized
  Dalecki{\u\i}-Kre{\u\i}n formula for higher order derivatives. We show that
  certain specializations of this formula lead to classical formulas
  of quantum perturbation theory.
  We show how our results are related to earlier results for higher
  order Fr\'echet derivatives.
  Block forms of complex step approximations are introduced and we show how
  those are related to evaluation of derivatives through the upper
  triangular form. These relations are illustrated with numerical
  examples.

\end{abstract}
\section{Introduction}
In this article, we are concerned with derivatives of matrix functions
on the form
\begin{equation}\label{eq:higher_deriv}
  \left.\frac{\partial^k}{\partial x_{d_1}\dots\partial x_{d_k}} f(A(x_1,\dots,x_j))\right|_{x_1=\dots=x_j=0}
\end{equation}
where $A(x)$ is a complex matrix in general nonlinear in $x =
(x_1,x_2,\dots,x_j)$ and $d_i \in \{1,\dots,j\},\ i=1,\dots,k$.  Here,
$f: \mathbb{C}^{n\times n} \rightarrow \mathbb{C}^{n\times n}$ is a
matrix function, generalized from a scalar function in the standard
sense~\cite{Higham_functions_of_matrices_2008}.
Several types of higher order derivatives have been considered
previously, and we will below discuss how they are related to each
other and to~\eqref{eq:higher_deriv}.
Besides intrinsic interest, our main motivation
to consider derivatives on the form in \eqref{eq:higher_deriv}
is that it includes quantum perturbation theory and
response calculations used to characterize material and molecular
properties.

The derivative \eqref{eq:higher_deriv} considered here can be seen as
a generalization of the G\^{a}teaux derivative
\begin{equation}
  G_f(A,E) = \frac{d}{dx}f(A+xE)|_{x=0}.
\end{equation}
In the first order case ($j=k=1$), the derivative \eqref{eq:higher_deriv} is
equal to
$G_f(A(0),A'(0))$. The Fr\'echet
derivative is the unique mapping $L_f(A,E)$, linear in $E$, that
satisfies
\begin{equation}
  L_f(A,E) = f(A+E) - f(A) + o(\|E\|).
\end{equation}
If the Fr\'echet derivative exists it is equal to the G\^{a}teaux
derivative~\cite{Higham_functions_of_matrices_2008} and then, in the first
order case, \eqref{eq:higher_deriv} is also equal to $L_f(A(0),A'(0))$.
A motivation to study the Fr\'echet derivative and methods for its
evaluation has been the use of the Fr\'echet derivative in matrix
function condition
estimation~\cite{condition_estimate_kenney_laub_1989,
  condition_estimate_mathias_1995}.  A key result is that
\begin{equation}\label{eq:mathias_intro}
  f\left(\begin{bmatrix}A & E \\ 0 & A\end{bmatrix}\right) =
  \begin{bmatrix}f(A) & L_f(A,E) \\ 0 & f(A)\end{bmatrix}
\end{equation}
where it is assumed that $f$ is $2m-1$ times continuously
differentiable on an open subset of $\mathbb{C}$ or $\mathbb{R}$
containing the eigenvalues of $A+xE$ and that the largest Jordan
block of $A+xE$ is at most $m$ for all $x$ in some neighborhood of
0~\cite{Mathias_derivative_1996,
  NajfeldHavel_derivative_1995}. Another important result, due to
Dalecki{\u\i} and Kre{\u\i}n~\cite{DaleckiiKrein_1965}, is that for Hermitian $A$ and $E$
and an eigendecomposition $A = V\Lambda V^*,
\Lambda=\textrm{diag}(\lambda_i)$, we have that
\begin{equation}\label{eq:daleckii_krein_intro}
  L_f(A,E) = V(G\circ (V^*EV))V^*
\end{equation}
where $G$ is the Loewner matrix with the divided differences
$f[\lambda_i,\lambda_j]$ as elements and $A\circ B$ is the Hadamard
product.
This result was extended to generalized matrix functions, defined in
terms of singular value decomposition, by
Noferini~\cite{noferini_2017}.
Krylov subspace methods to compute the action of the Fr\'echet
derivative on a vector were proposed
in~\cite{kandolf_krylov_frechet_2021, kandolf_krylov_frechet_2017}.
There are a number of studies that address Fr\'echet derivatives
of particular matrix functions, including the matrix
exponential~\cite{mathias_exponential_1992,
  NajfeldHavel_derivative_1995, almohy_exp_frechet}, the matrix
cosine~\cite{al_mohy_cosine_2022}, and the matrix $p$th
root~\cite{cardoso_pthroot_2011}.

To also compute the sensitivity of the condition
estimates---the condition number of the condition number---has been a
motivation for work on higher order Fr\'echet
derivatives~\cite{HighamRelton_Frechet_2014}.
The Fr\'echet derivative of order $j>1$ can be defined as the
$j$-linear mapping $L_f^{(j)}(A_0,E_1,\dots,E_j)$, linear in each of
$E_1,\dots,E_j$, that satisfies
\begin{multline}\label{eq:frechet_higher_def}
  L_f^{(j)}(A_0,E_1,\dots,E_j) =    
  L_f^{(j-1)}(A_0+E_j,E_1,\dots,E_{j-1}) \\ -L_f^{(j-1)}(A_0,E_1,\dots,E_{j-1})
  +o(\|E_j\|),
\end{multline}
where $L_f^{(1)}(A,E) = L_f(A,E)$.
We may use the shorthand notation $L_f^{(j)}(A,E)$ when $E_i = E,\, i
= 1,\dots, j$.
If the higher order Fr\'echet derivative is continuous in $A_0$, then
\begin{align}\label{eq:higher_linear_partial}
  L_f^{(j)}(A_0,E_1,\dots,E_j) = &
  \left.\frac{\partial^j}{\partial x_{1}\dots\partial x_{j}} f\left(A_0+\sum_{i=1}^jx_iE_i\right)\right|_{x_1=\dots=x_j=0}
\end{align}
which is a special case of \eqref{eq:higher_deriv} with $A(x) =
A_0+\sum_{i=1}^jx_iE_i$, see \cite{HighamRelton_Frechet_2014,Nashed_differentials_1966}.
An integral representation of higher order Fr\'echet derivatives was
proposed in \cite{Schweitzer_2023}.  Complex step approximations of
first and higher order Fr\'echet derivatives were proposed
in~\cite{complex_step_higher_frechet_2021,
  Al-MohyHigham_derivative_2009, werner_arxiv_2022}.
Higham and Relton~\cite{HighamRelton_Frechet_2014} generalized~\eqref{eq:mathias_intro} to higher order
Fr\'echet derivatives. Mathias~\cite{Mathias_derivative_1996} generalized~\eqref{eq:mathias_intro} to
higher derivatives on the form 
\begin{equation}\label{eq:higher_nonlinear_ordinary}
  \left.\frac{d^j}{dx^j}f(A(x))\right|_{x=0},
\end{equation}
where $A(x)$ is a matrix depending on a single variable
$x$.
A similar development was carried out by Najfeld and
Havel~\cite{NajfeldHavel_derivative_1995} for the higher derivative
\begin{equation}\label{eq:higher_linear_ordinary}
  \left.\frac{d^j}{dx^j}f(A_0 +xE_1)\right|_{x=0}.
\end{equation}
This type of derivative was also used in error estimates for the
complex step approximation by Al-Mohy and
Higham~\cite{Al-MohyHigham_derivative_2009}.
An explicit expression to
evaluate~\eqref{eq:higher_nonlinear_ordinary} for $j=2$, involving
Hadamard products and divided differences similar
to~\eqref{eq:daleckii_krein_intro}, was presented
in~\cite{HornJohnson_topics_1991}.
We note
that~\eqref{eq:higher_linear_ordinary} is a special case both of
\eqref{eq:higher_nonlinear_ordinary} with $A(x) = A_0 +xE_1$ and
of~\eqref{eq:higher_linear_partial} with $E_1 = \dots = E_j$,
see~\cite{Nashed_differentials_1966, Zorn_analytic_functions_1945}.
All of the above formulations are special
cases of \eqref{eq:higher_deriv}, which includes both the possibility
of partial derivatives and nonlinearities in $A(x)$.
We note that nonlinearities in $A(x)$ are always present in
self-consistent quantum response calculations.  Even if an external
perturbation is linear, it will induce a nonlinear perturbation
through the self-consistent field equations, see for
example~\cite{weber_2005}.

The purpose of this work is to develop theory for higher order
derivatives of matrix functions, that can be useful for method
development and to understand existing methods and the relationship
between them.
In particular we develop a block upper triangular formulation for
higher order derivatives, building on the result
in~\eqref{eq:mathias_intro} for the first order derivative. We show
how this formulation is related to results for higher order Fr\'echet
derivatives. By a characterization of the Jordan structure of the
block upper triangular form, we derive conditions for the existence of
higher order derivatives. Those conditions are specialized to higher
order Fr\'echet derivatives, giving a weaker requirement on the
regularity of $f$ than in previous
work~\cite{HighamRelton_Frechet_2014}.
We derive also a generalized Dalecki{\u\i}-Kre{\u\i}n formula and show that
certain specializations thereof result in classical formulas 
of quantum perturbation theory.
We show how the block upper triangular matrix formulation for higher
derivatives is related to complex step approximations~\cite{Al-MohyHigham_derivative_2009, complex_step_higher_frechet_2021,  werner_arxiv_2022}. 

\section{Existence of higher order derivatives}\label{sec:existence}
Let $\Omega$ be an open subset of $\mathbb{C}$ and let $M_n(\Omega,m)$
denote the set of $n\times n$ complex matrices that have spectrum in
$\Omega$ and largest Jordan block of size not exceeding $m$.
The existence of higher order derivatives has previously been
considered in several works. In particular, Higham and Relton~\cite{HighamRelton_Frechet_2014} showed
that a sufficient condition for the existence of the Fr\'echet
derivative of order $j$, $L_f^{(j)}(A)$ with $A\in M_n(\Omega,m)$, is
that $f$ is $2^jm-1$ times continuously differentiable on $\Omega$,
but also conjectured that it may be possible to relax this
condition.  Mathias~\cite{Mathias_derivative_1996} showed that for
the existence of $\frac{d^j}{dt^j}f(A(t))|_{t=t_0}$,
with $A(t)$ $j$ times differentiable at $t_0$ and $A(t)\in M_n(\Omega,m)$
for all $t$ in some neighborhood of $t_0$,
it is sufficient that $f$
is $(j+1)m-1$ times continuously differentiable on
$\Omega$.
We will show that such a result holds also for the higher order
Fr\'echet derivative and the general case of~\eqref{eq:higher_deriv}.

\begin{theorem}\label{thm:exist_higher_deriv}
  Let $A(x)$ be $k$ times differentiable at $\bar{x}=(\bar{x}_1,\dots,\bar{x}_j)$ and
  assume that $A(x) \in M_n(\Omega,m)$ for all $x$ in some neighborhood of
  $\bar{x}$. Let
  \begin{equation} \label{eq:X_sequence}
    \begin{aligned}
      X_0(x) & = A(x),  \\
      X_i(x) & = 
      \begin{bmatrix}
        X_{i-1}(x) & \frac{\partial}{\partial x_{d_i}} X_{i-1}(x) \\
        0 & X_{i-1}(x)
      \end{bmatrix},\quad d_i \in \{1,\dots , j\}, \quad i = 1,\dots,k. 
    \end{aligned}
  \end{equation}
  Let $f$ be $(k+1)m-1$ times continuously differentiable on $\Omega$.
  Then, $f(A(x))$ is $k$ times differentiable at $\bar{x}$ and 
  $F_k(\bar{x}) = f(X_k(\bar{x}))$, where
  \begin{equation}
    \begin{aligned}
      F_0(x) & = f(A(x)), \\
      F_i(x) & = 
      \begin{bmatrix}
        F_{i-1}(x) & \frac{\partial}{\partial x_{d_i}} F_{i-1}(x) \\
        0 & F_{i-1}(x)
      \end{bmatrix}, \quad i = 1,\dots,k.
    \end{aligned}
  \end{equation}
\end{theorem}

Theorem~\ref{thm:exist_higher_deriv} implies the perhaps expected
result that for $A(x)$ and $B(x)$ with identical partial derivatives
at $\bar{x}$ up to some order $k$, the corresponding $k$ partial
derivatives of $f(A(x))$ and $f(B(x))$ at $\bar{x}$ are identical. In
particular, $A(x)$ may be replaced with a power series expansion
around $\bar{x}$ up to some desired order.  For example with
$x=(x_1,x_2)$ and
\begin{equation}
  B(x) = A(0) + x_1A_{x_1}(0) + x_2A_{x_2}(0) + x_1x_2A_{x_1x_2}(0)  
\end{equation}
we have that $\frac{\partial^2}{\partial x_1 \partial
  x_2}f(A(x))|_{x=0} = \frac{\partial^2}{\partial x_1 \partial
  x_2}f(B(x))|_{x=0}$.  This type of truncated power series
representation of $A(x)$ is typically used in quantum response
calculations, see for example~\cite{weber_2005}.

A crucial part of the proof of Theorem~\ref{thm:exist_higher_deriv} is
the existence of $f(X_k(\bar{x}))$.
To this end, we need a characterization of the Jordan structure of
$X_k(\bar{x})$.
We will use the result below by Friedland and Hershkowitz~\cite{Friedland_1988}. In the
following, $p(G)$ denotes the number of vertices along the longest
path in the graph represented by the adjacency matrix $G$.

\begin{lemma}[Friedland and Hershkowitz~\cite{Friedland_1988}]\label{lemma:hers}
  Let $A$ be an upper block triangular complex valued matrix.
  Associate with $A$ its reduced directed acyclic graph represented by
  the adjacency matrix $G$ such that
  \begin{equation}
    G_{ij} = \begin{cases}
      1 & \textrm{if the submatrix } A_{ij} \neq 0 \textrm{ and } i\neq j, \\
      0 & \textrm{otherwise.}
  \end{cases}
  \end{equation}
 Then, the index of $A$ does not exceed the maximal
  sum of the indices of $A_{ii},\dots,A_{jj}$ along all possible paths
  $(i,\dots , j)$ in the graph of $G$.
\end{lemma}
We note that several variants of this result exist in the
literature~\cite{meyer_index_1977, HornJohnson_matrix_analysis_2013,
  Mathias_derivative_1996} that all imply only the weaker result that the
index of $A$ is less than or equal to the sum of the indices of the
diagonal blocks of $A$.
Such a result is sufficient to obtain the desired result when $A(t)$
is a function of one variable since then, an upper triangular block
matrix with $(j+1) \times (j+1)$ blocks may be used for the
characterization of a $j$th order derivative,
see~\cite{Mathias_derivative_1996}.
With further assumptions on $A$ even stronger
results may be possible, see for example~\cite{bru_index_1995,
  johnson_index_1990}, but this possibility will not be further
considered here.

\begin{lemma}\label{lemma:longest_path}
  Consider a directed acyclic graph with strictly upper triangular
  adjacency matrix $G = \begin{bmatrix} A & I+C \\ 0 &
    B \end{bmatrix}$ where $A$ and $B$ represent subgraphs of equal
  order. Assume that $C$ is strictly upper triangular.  Consider also
  the reduced graph associated with $R = \begin{bmatrix} A & I \\ 0 &
    B \end{bmatrix}$. The following holds:
  \begin{itemize}
  \item[(i)] If $B=C$, then $p(G) = p(R)$. 
  \item[(ii)] If $A=B$, then $p(R) = p(A)+1$.
  \item[(iii)] If $A=B=C$, then $p(G) = p(A)+1$.
  \end{itemize}
\end{lemma}

\begin{proof}
  Let $\{A_i\}_{i=1}^n$ and $\{B_i\}_{i=1}^n$ denote the vertices of
  $A$ and $B$, respectively.
  \begin{itemize}
  \item[(i)] Assume that the longest path includes a path $(A_i,B_j)$
    with $j>i$, corresponding to a nonzero element in $C$.  Since
    $B=C$ there exists a path $(A_i,B_i,B_j)$, which is a
    contradiction. Thus, the longest path does not include any path
    represented in $C$ and (i) follows.
  \item[(ii)]
    Let $(A_i,\dots,A_j)$ be the
    longest path in $A$ with length $p(A)$. Clearly there is a path
    $(A_i,\dots,A_j,B_j)$ in $R$ with length $p(A)+1$. Assume that there
    is a path $(A_i,\dots,A_j,B_j,B_k,\dots,B_l)$ in $R$ with length $r
    > p(A)+1$. Then, since $A$ and $B$ are identical, there is a path
    $(A_i,\dots,A_j,A_k,\dots,A_l)$ in $A$ with length $r-1 > p(A)$,
    which is a contradiction. Thus, (ii) follows.
  \item[(iii)] (iii) follows from (i) and (ii). 
  \end{itemize}
  \end{proof}

\begin{lemma}\label{lemma:index_of_Xk_new}  
  For all $x = (x_1,\dots,x_j)$ in some neighborhood of $\bar{x} = (\bar{x}_1,\dots,\bar{x}_j)$, assume
  that $A(x) \in M_n(\Omega,m)$ and let $A(x)$ be $k$ times
  differentiable. Let
  \begin{align}
    X_0(x) & = A(x), \\
    Q_i(x,\epsilon) & =
    \begin{pmatrix}
      X_{i-1}(x_1,\dots, x_j) & \frac{X_{i-1}(x_1, \dots, x_{d_i}+\epsilon,\dots, x_j) - X_{i-1}(x_1,\dots, x_j)}{\epsilon}  \\
      0 & X_{i-1}(x_1, \dots, x_{d_i}+\epsilon,\dots, x_j)
    \end{pmatrix},\\
    X_i(x) & = \lim_{\epsilon \rightarrow 0} Q_i(x,\epsilon),\quad d_i \in \{1,\dots , j\}, \quad i = 1,\dots,k.
  \end{align}
  Then, 
  \begin{itemize}
  \item[(i)] for sufficiently small $\epsilon$, $Q_i(\bar{x},\epsilon) \in M_{2^in}(\Omega,(i+1)m)$
  \item[(ii)] $X_i(\bar{x}) \in M_{2^in}(\Omega,(i+1)m),\ i = 1,\dots,k$.
  \end{itemize}
\end{lemma}
\begin{proof}
  The graphs (defined in Lemma~\ref{lemma:hers}) associated with
  $Q_i(\bar{x},\epsilon)$ and $X_i(\bar{x})$, of orders $2^i$,
  $i=0,\dots,k$ are represented by the adjacency matrices $G_i$
  defined by the recursion
  \begin{equation}
    \begin{aligned}
      G_0 & = 0 \\
      G_i & = 
      \begin{bmatrix}
        G_{i-1} & I+G_{i-1} \\
        0 & G_{i-1}
      \end{bmatrix}, \quad i = 1,\dots,k.
    \end{aligned}
  \end{equation}
  By induction and Lemma~\ref{lemma:longest_path} we have that $p(G_i)
  = i$.  The diagonal blocks of both $X_i(\bar{x})$ and $Q_i(\bar{x},\epsilon)$
  have Jordan blocks of size not exceeding $m$, for sufficiently small
  $\epsilon$. Therefore, we have by Lemma~\ref{lemma:hers} that the
  size of the largest Jordan blocks of $Q_i(\bar{x},\epsilon)$ and
  $X_i(\bar{x})$ do not exceed $(p(G_i)+1)m = (i+1)m$.
\end{proof}

\begin{proof}[Proof of Theorem~\ref{thm:exist_higher_deriv}]
  This proof consists essentially of two parts: 1) the existence of
  $f(X_k(\bar{x}))$ and 2) its equality with $F_k(\bar{x})$.
  
  The first part makes use of the characterization of the Jordan
  structure provided by Lemma~\ref{lemma:index_of_Xk_new}.  By
  Lemma~\ref{lemma:index_of_Xk_new}, the size of the largest Jordan
  block of $X_k(\bar{x})$ and $Q_k(\bar{x},\epsilon)$ is at most
  $(k+1)m$. Therefore, and since $f$ is $(k+1)m-1$ times continuously
  differentiable $f(X_k(\bar{x}))$ and $f(Q_k(\bar{x},\epsilon))$
  exist.

  The second part follows closely the proof
  of~\cite[Theorem~2.1]{Mathias_derivative_1996}, but makes use of
  induction to go to higher order derivatives.  We assume that the
  result holds for $k = i-1$, i.e.\ that $F_{i-1}(\bar{x}) =
  f(X_{i-1}(\bar{x}))$.
  We follow Mathias~\cite{Mathias_derivative_1996} and define 
  \begin{align}
    S =
    \begin{pmatrix}
      I & \epsilon^{-1}I \\
      0 & I
    \end{pmatrix}
  \end{align}
  with $\epsilon\neq 0$. Then
  \begin{align}
    f(X_i(\bar{x}))
    & =
    f( \lim_{\epsilon \rightarrow 0} Q_i(\bar{x},\epsilon))
    \\ & =
    \lim_{\epsilon \rightarrow 0} f( Q_i(\bar{x},\epsilon))
    \label{eq:lim_switch}
    \\ & =
    \lim_{\epsilon \rightarrow 0} 
    f \begin{pmatrix}
      X_{i-1}(\bar{x}) & \frac{X_{i-1}(\bar{x}_1, \dots, \bar{x}_{d_i}+\epsilon, \dots, \bar{x}_j) - X_{i-1}(\bar{x})}{\epsilon}  \\
      0 & X_{i-1}(\bar{x}_1, \dots, \bar{x}_{d_i}+\epsilon, \dots, \bar{x}_j)
    \end{pmatrix}
    \\ & = 
    \lim_{\epsilon \rightarrow 0} 
    S f \left(S^{-1}\begin{pmatrix}
      X_{i-1}(\bar{x}) & \frac{X_{i-1}(\bar{x}_1, \dots, \bar{x}_{d_i}+\epsilon, \dots, \bar{x}_j) - X_{i-1}(\bar{x})}{\epsilon}  \\
      0 & X_{i-1}(\bar{x}_1, \dots, \bar{x}_{d_i}+\epsilon, \dots, \bar{x}_j)
  \end{pmatrix}S\right)S^{-1} 
  \\ & = 
  \lim_{\epsilon \rightarrow 0} 
  S f \begin{pmatrix}
    X_{i-1}(\bar{x}) & 0  \\
    0 & X_{i-1}(\bar{x}_1, \dots, \bar{x}_{d_i}+\epsilon, \dots, \bar{x}_j)
  \end{pmatrix}S^{-1} 
  \\ & = 
  \lim_{\epsilon \rightarrow 0} 
  S \begin{pmatrix}
    f(X_{i-1}(\bar{x})) & 0  \\
    0 & f(X_{i-1}(\bar{x}_1, \dots, \bar{x}_{d_i}+\epsilon, \dots, \bar{x}_j))
  \end{pmatrix}S^{-1} 
  \\ & =
  \lim_{\epsilon \rightarrow 0} 
  \begin{pmatrix}
    f(X_{i-1}(\bar{x})) & \frac{f(X_{i-1}(\bar{x}_1, \dots, \bar{x}_{d_i}+\epsilon, \dots, \bar{x}_j)) - f(X_{i-1}(\bar{x}))}{\epsilon}  \\
    0 & f(X_{i-1}(\bar{x}_1, \dots, \bar{x}_{d_i}+\epsilon, \dots, \bar{x}_j))
  \end{pmatrix}
  \\ & =
  \begin{pmatrix}
    f(X_{i-1}(\bar{x})) & \left.\frac{\partial}{\partial x_{d_i}} f(X_{i-1}(x))\right|_{x=\bar{x}}   \\
    0 & f(X_{i-1}(\bar{x}))
  \end{pmatrix}
  \\ & =
  \begin{pmatrix}
    F_{i-1}(\bar{x}) & \left.\frac{\partial}{\partial x_{d_i}} F_{i-1}(x)\right|_{x=\bar{x}} \\
    0 & F_{i-1}(\bar{x})
  \end{pmatrix}
  \\ & =
  F_i(\bar{x})
  \end{align}
  The equality in \eqref{eq:lim_switch} is due to $f$ being continuous
  on $M_{2^in}(\Omega,(i+1)m)$, see~\cite[Lemma~1.1]{Mathias_derivative_1996}.
  The base case $k=0$ is trivially true.
\end{proof}

\section{Relation to higher order Fr\'echet derivatives}

\begin{theorem}\label{thm:exist_frechet}
  Assume that $A(x)=A_0+\sum_{i=1}^kx_iE_i \in M_n(\Omega,m)$ for all $x=(x_1,\dots,x_k)$ in some neighborhood of $x=0$
  and let $f$ be $(k+1)m-1$ times continuously
  differentiable on $\Omega$. Then, the $k$th Fr\'echet derivative
  $L_f^{(k)}(A_0,E_1,\dots,E_k)$ exists, is continuous in $A_0$ and
  $E_1,\dots,E_k$, and is given by the upper right $n\times n$ block
  of $\left.f(X_k(x))\right|_{x=0}$
  where, for $i = 0,\dots, k$, $X_i(x)$ is defined as in \eqref{eq:X_sequence}
  and $d_i = i$. 
\end{theorem}
\begin{proof}
  This is essentially a strengthening
  of~\cite[Theorem~3.5]{HighamRelton_Frechet_2014} which requires
  $2^km-1$ continuous derivatives of $f$. The result follows from the
  original proof with the characterization of the Jordan structure
  provided by Lemma~\ref{lemma:index_of_Xk_new} of the present
  work. With $A(x)=A_0+\sum_{i=1}^kx_iE_i$ and $d_i=i,\ i=1,\dots,k$,
  the sequence of matrices $X_i(x),\ i = 0,\dots,k$ is given by
  \begin{equation}\label{eq:recursive_frechet}
    \begin{aligned}
      X_0(x) & = A_0+\sum_{i=1}^kx_iE_i,  \\
      X_i(x) & = 
      \begin{bmatrix}
        X_{i-1}(x) & I_{2^{i-1}} \otimes E_i  \\
        0 & X_{i-1}(x)
      \end{bmatrix}, \quad i = 1,\dots,k.
    \end{aligned}
  \end{equation}
    which is used in (3.3) of~\cite{HighamRelton_Frechet_2014} (with
    $x=0$).  By Theorem~\ref{thm:exist_higher_deriv}, the upper right
    $n\times n$ block of $f(X_k(0))$ is
    \begin{align}
      \left[f(X_k(0))\right]_{1,k} = \left.\frac{\partial^k}{\partial x_1 \dots \partial x_k} f(A(x))\right|_{x=0}.
    \end{align}
    Since the size of the largest Jordan block in $X_k(0)$ is at most
    $(k+1)m$, we have that $(k+1)m-1$ continuous derivatives of $f$ is
    sufficient for the existence and continuity of $f(X_k(0))$.
\end{proof}

\begin{lemma}\label{lem:lemma_deriv_of_Lf}
  Let $A(x)$ and $E_i(x),\ i=1,\dots,k$ be
  in $M_n(\Omega,m)$ for all $x$ in some neighborhood of
  $\bar{x}=(\bar{x}_1,\dots,\bar{x}_j)$ and differentiable at
  $\bar{x}$.
  Assume that $f$ is $(k+2)m-1$ times continuously differentiable on $\Omega$.
  Then the partial derivative 
  \begin{multline}  
    {\textstyle\frac{\partial}{\partial x_d}} \left.L_f^{(k)}(A(x),E_1(x),\dots,E_k(x))\right|_{x=\bar{x}}
    = \\
    \begin{aligned}
    & L_f^{(k+1)}(A(\bar{x}),E_1(\bar{x}),\dots,E_k(\bar{x}),{\textstyle\frac{\partial}{\partial x_d}} A(x)|_{x=\bar{x}})
    \\ & +
    \sum_{i=1}^k L_f^{(k)}(A(\bar{x}),E_1(\bar{x}),\dots, E_{i-1}(\bar{x}),{\textstyle\frac{\partial}{\partial x_d}} E_i(x)|_{x=\bar{x}},E_{i+1}(\bar{x}),\dots,E_k(\bar{x})).
    \end{aligned}
  \end{multline}
\end{lemma}
\begin{proof} 
  \begin{align}
    & \frac{\partial}{\partial x_d} \left.(L_f^{(k)}(A(x),E_1(x),\dots,E_k(x))\right|_{x=\bar{x}}
    \\
      & 
      \begin{aligned}[t]
        = \lim_{h\rightarrow 0} \frac{1}{h}(&
        \begin{aligned}[t]L_f^{(k)}(&A(\bar{x})+h{\textstyle \frac{\partial}{\partial x_d}}A(x)|_{x=\bar{x}}+o(h),
          \\& E_1(\bar{x})+h{\textstyle\frac{\partial}{\partial x_d}}E_1(x)|_{x=\bar{x}}+o(h),
          \dots,
          \\& E_k(\bar{x})+h{\textstyle\frac{\partial}{\partial x_d}}E_k(x)|_{x=\bar{x}}+o(h))
        \end{aligned}
        \\
        &-L_f^{(k)}(A(\bar{x}),E_1(\bar{x}),\dots,E_k(\bar{x}))) 
      \end{aligned}
      \\
      & 
      \begin{aligned}[t]
        = \lim_{h\rightarrow 0} \frac{1}{h}(&
        \begin{aligned}[t]L_f^{(k)}(&A(\bar{x}),E_1(\bar{x})+h{\textstyle\frac{\partial}{\partial x_d}}E_1(x)|_{x=\bar{x}}+o(h),\dots,
          \\& E_k(\bar{x})+h{\textstyle\frac{\partial}{\partial x_d}}E_k(x)|_{x=\bar{x}}+o(h))
        \end{aligned}
        \\
        &+
        \begin{aligned}[t]
          L_f^{(k+1)}(&A(\bar{x}),E_1(\bar{x})+h{\textstyle\frac{\partial}{\partial x_d}}E_1(x)|_{x=\bar{x}}+o(h),\dots,
          \\&
          E_k(\bar{x})+h{\textstyle\frac{\partial}{\partial x_d}}E_k(x)|_{x=\bar{x}}+o(h), h{\textstyle \frac{\partial}{\partial x_d}}A(x)|_{x=\bar{x}}+o(h))
        \end{aligned}          
        \\
        &+
        o(h\|{\textstyle \frac{\partial}{\partial x_d}}A(x)|_{x=\bar{x}}\|)
        \\
        &-L_f^{(k)}(A(\bar{x}),E_1(\bar{x}),\dots,E_k(\bar{x})))
      \end{aligned}
      \\
      & 
      \begin{aligned}[t]
        = &\ L_f^{(k+1)}(A(\bar{x}),E_1(\bar{x}),\dots,E_k(\bar{x}),{\textstyle\frac{\partial}{\partial x_d}} A(x)|_{x=\bar{x}})
        \\ & +
        \sum_{i=1}^k L_f^{(k)}(A(\bar{x}),E_1(\bar{x}),\dots, E_{i-1}(\bar{x}),{\textstyle\frac{\partial}{\partial x_d}} E_i(x)|_{x=\bar{x}},E_{i+1}(\bar{x}),\dots,E_k(\bar{x}))      
      \end{aligned}
  \end{align}
  Here we used the definition and the linearity of the higher order
  Fr\'echet derivative~\eqref{eq:frechet_higher_def}.  
\end{proof}

We define the multi-set $S_\alpha^k$ of partitions of a multi-index
$\alpha$ into $k$ nonzero multi-indices with the following
recursion. For $1 < k \leq |\alpha|$, partition $\alpha = \beta +
\gamma$ with $|\gamma| = 1$ and let
\begin{align}\label{eq:S_definition}
  S_{\alpha}^k = \Big\{& s:\, s=\{u_1,\dots,u_{k-1},\gamma\},\, u \in S_\beta^{k-1} \textrm{ and }\\& s:\, s=\{u_1,\dots,u_{j}+\gamma,\dots,u_{k},\}_{j=1}^k,\, u \in S_\beta^{k}\Big\}\nonumber
\end{align}
where $S_\alpha^0 = \emptyset$, $S_\alpha^1 = \{\alpha\}$, and
$S_\alpha^k = \emptyset$ for $k>|\alpha|$.  Note that the recursion
may generate duplicate partitions that are all members, making
$S_\alpha^k$ a multi-set.  The partitions themselves are also
multi-sets; order within a partition does not matter here.  For each
partition $s$ in $S_\alpha^k$, let $s_1,\dots,s_k$ be an arbitrary
enumeration of its members. The members of a partition fulfil $s_i \leq \alpha $ and
$|s_i|>0$, for $i = 1\dots k$ and $\sum_{i=1}^k s_i = \alpha$.
As an example, with $\alpha=(2,1)$ and $k=2$ we get
\begin{align} \label{eq:S_set_example}
  S_{(2,1)}^2 = \Big\{ & \big\{ (2,0),(0,1) \big\}, \\
                     & \big\{(1,1),(1,0) \big\}, \nonumber\\
                     & \big\{(1,1),(1,0) \big\}  \nonumber
              \Big\}.
\end{align}

For convenience, we will sometimes use the shorthand notation $A^{(\alpha)} =
A^{(\alpha)}(\bar{x}) = \partial^\alpha A(x)|_{x=\bar{x}}$ and
$A=A(\bar{x})$.

\begin{theorem}\label{thm:derivative_as_linear_combination_of_frechet}
  Let $A(x)$ be $k$ times differentiable at $\bar{x}=(\bar{x}_1,\dots,\bar{x}_j)$ and
  assume that $A(x) \in M_n(\Omega,m)$ for all $x$ in some neighborhood of
  $\bar{x}$. 
  Let $f$ be $(k+1)m-1$ times continuously differentiable on $\Omega$ and assume that $|\alpha|\leq k$. Then,
  \begin{align}
    \partial^\alpha f(A(x))|_{x=\bar{x}} = \sum_{i=1}^{|\alpha|} \sum_{s\in S_\alpha^i} L_f^{(i)}(A(\bar{x}),A^{(s_1)}(\bar{x}),\dots,A^{(s_i)}(\bar{x}))
  \end{align}
\end{theorem}
\begin{proof}
  Assume that $|\alpha|>1$ and partition $\alpha = \beta + \gamma$ with $|\gamma| = 1$, and assume that the result holds for $\left.\partial^\beta f(A(x))\right|_{x=\bar{x}}$. Then,
  \begin{align}
    \left.\partial^\alpha f(A(x))\right|_{x=\bar{x}}
    &= \left.\partial^\gamma\partial^\beta f(A(x))\right|_{x=\bar{x}}
    \\&=
    \partial^\gamma
    \sum_{i=1}^{|\beta|} \sum_{s\in S_\beta^i} L_f^{(i)}(A(x),A^{(s_1)}(x),\dots,A^{(s_i)}(x))\Big{|}_{x=\bar{x}}
    \\&
    \begin{aligned}[t]
      =\sum_{i=1}^{|\beta|} \sum_{s\in S_\beta^i} & L_f^{(i+1)}(A,A^{(s_1)},\dots,A^{(s_i)},A^{(\gamma)})
      \\&
      +\sum_{j=1}^i L_f^{(i)}(A,A^{(s_1)},\dots,A^{(s_j+\gamma)},\dots,
      A^{(s_i)})
    \end{aligned}\label{eq:use_lemma_deriv_of_Lf}
    \\&
    \begin{aligned}[t]
      =\sum_{i=1}^{|\beta|+1} & \sum_{s\in S_\beta^{i-1}} L_f^{(i)}(A,A^{(s_1)},\dots,A^{(s_{i-1})},A^{(\gamma)})
      \\&
      +\sum_{s\in S_\beta^{i}}\sum_{j=1}^i L_f^{(i)}(A,A^{(s_1)},\dots,A^{(s_j+\gamma)},\dots,
      A^{(s_i)})
    \end{aligned}\label{eq:shift_index}
    \\&
    =\sum_{i=1}^{|\alpha|} \sum_{s\in S_\alpha^i} L_f^{(i)}(A,A^{(s_1)},\dots,A^{(s_i)})
  \end{align}
  In~\eqref{eq:use_lemma_deriv_of_Lf} we used
  Lemma~\ref{lem:lemma_deriv_of_Lf}. In~\eqref{eq:shift_index}, note that $S_\beta^0 = S_\beta^{|\beta|+1} = \emptyset$.
  Finally, we used the recursive definition of $S_{\alpha}^i$ in
  \eqref{eq:S_definition}.
  The base case with $|\alpha|=1$ reduces to $\partial^\alpha
  f(A(x))|_{x=\bar{x}} = L_f^{(1)}(A(\bar{x}),A^{(\alpha)}(\bar{x}))$.
\end{proof}

\section{A generalized Dalecki{\u\i}-Kre{\u\i}n formula}
The block upper triangular formulation in Theorem~\ref{thm:exist_higher_deriv} opens the way to
numerous possibilities for the evaluation of higher order derivatives
of matrix functions. In particular, the formulation can be used for
alternative derivation of a number of existing methods, and may be
useful for their further development and analysis. As an example, we
take here the route via Schur decomposition and the explicit formula
for the elements of matrix functions of triangular matrices
provided by Descloux~\cite{Descloux_1963, davis_1973}.
The entries of a matrix function of an upper triangular matrix $U$ is
given by $(f(U))_{ii} = f(\lambda_i), i = 1,\dots, n$ and for $j>i$,
\begin{multline}\label{eq:descloux}
  (f(U))_{ij} = \\ \sum_{m=1}^{j-i} \sum_{i < k_1 < k_2 < \dots < k_{m-1} < j} U_{i,k_1}U_{k_1,k_2}\dots U_{k_{m-1},j}f[\lambda_i,\lambda_{k_1},\dots,\lambda_{k_{m-1}},\lambda_{j}]
\end{multline}
where $f[x_1,x_2,\dots,x_k]$ is the $k$th order divided difference and $\{\lambda_i\}$ are the eigenvalues and diagonal entries of $U$.
The matrix $U$ can be seen as the adjacency matrix of a directed
acyclic graph. The second sum above runs over all possible paths
between vertex $i$ and vertex $j$ that pass exactly
$m$ edges along the way.
The first sum runs over all
possible path lengths between vertex $i$ and vertex $j$. 
In general, the cost of evaluating the above expression is too high
for it to be of practical use. However, in certain cases with zero
structure within the upper triangle of $U$ the expression may be
useful, provided that the zeros sufficiently limits the number and/or
lengths of paths in the graph.

Here, we will need the multi-set $T_\alpha^k$ containing all possible
permutations of each partition in $S_\alpha^k$ introduced in the
previous section. $T_\alpha^k$ is a multi-set of $k$-tuples and has
$|T_\alpha^k| = k!|S_\alpha^k|$ members. In contrast to $S_\alpha^k$,
order within a partition does matter in this case. For the example
with $\alpha=(2,1)$ and $k=2$ considered in \eqref{eq:S_set_example}
we get that
\begin{align}
  T_{(2,1)}^2 = \Big\{ & \big((2,0),(0,1)\big) , \big((0,1),(2,0)\big), \\
                     & \big((1,1),(1,0)\big),\big((1,0),(1,1)\big), \nonumber\\
                     & \big((1,1),(1,0)\big),\big((1,0),(1,1)\big) \nonumber  
              \Big\}.
\end{align}

We consider first the general case.  Consider a $k$th order derivative
and the sequence of matrices $X_i, i = 0,\dots, k$ provided in
Theorem~\ref{thm:exist_higher_deriv}.  We let $U = Q^*A(\bar{x})Q$ be
a Schur form of $A(\bar{x})$. Then, $U_k =
\diag_k(Q^*)X_k(\bar{x})\diag_k(Q)$ is a Schur form of $X_k(\bar{x})$,
where $\diag_k(Q)$ is the block diagonal matrix where $Q$ is repeated
$2^k$ times on the diagonal.  The $k$th order derivative is given by
$Q\left((f(U_k))_{1,2^k}\right)Q^*$ and can be computed using the QR
algorithm for the Schur form of $A(\bar{x})$ and the Descloux formula
\eqref{eq:descloux} for the upper right $n\times n$ block of $f(U_k)$.

An important special case is when $A(\bar{x})$ is Hermitian which
makes $U$ and the diagonal blocks of $U_k$ diagonal. In this case, the
number of edges along any path in $U_k$ is bounded by $k$ and we get that  
\begin{multline}\label{eq:daleckii_krein_generalized}
  \left(Q^*(\partial^\alpha f(A(x))|_{x=\bar{x}})Q\right)_{i,j} = \\ \sum_{m=1}^{|\alpha|}\sum_{t\in T_{\alpha}^m} \sum_{k_1=1}^n \dots \sum_{k_{m-1}=1}^n
  U_{i,k_1}^{(t_1)} U_{k_1,k_2}^{(t_2)}\dots U_{k_{m-1},j}^{(t_{m})}f[\lambda_i,\lambda_{k_1},\dots,\lambda_{k_{m-1}},\lambda_j]
\end{multline}
where $U^{(\alpha)} = Q^*A^{(\alpha)}Q$ and $\lambda_i,\, i=1,\dots,
n$ are the eigenvalues of $A$. Note that $U^{(\alpha)}$ are usually
not upper triangular but are blocks of the upper triangular matrix
$U_k$. This can be seen as a generalization of the Dalecki{\u\i}-Kre{\u\i}n
formula \eqref{eq:daleckii_krein_intro} for first order~\cite{DaleckiiKrein_1965}, 
in the present context given by 
\begin{align}\label{eq:daleckii_krein}
  (Q^*(\partial^\alpha f(A(x))|_{x=\bar{x}})Q)_{i,j} & = U^{(\alpha)}_{i,j}f[\lambda_i,\lambda_j], \quad \textrm{with } |\alpha | = 1.
\end{align}
In the case of second order with $\alpha = \beta + \gamma$, $|\beta| =
|\gamma| = 1$, we get
\begin{multline}\label{eq:second_order_explicit}
  (Q^*(\partial^\alpha f(A(x))|_{x=\bar{x}})Q)_{i,j} = \\
  U^{(\alpha)}_{i,j}f[\lambda_i,\lambda_j] + 
  \sum_{k=1}^n U_{i,k}^{(\beta)} U_{k,j}^{(\gamma)} f[\lambda_i,\lambda_k,\lambda_j] +
  \sum_{k=1}^n U_{i,k}^{(\gamma)} U_{k,j}^{(\beta)} f[\lambda_i,\lambda_k,\lambda_j]
\end{multline}
which has previously been presented in~\cite{HornJohnson_topics_1991}
for the non-mixed $\beta = \gamma$ case.
We show in Appendix~\ref{app:perturbation_theory} that further specialization of the above
formulas lead to classical formulas for time-independent perturbation
theory of quantum mechanics. 
We expect that other specializations will be useful in future
development of computational methods for higher order response
calculations, for example using alternative distribution
functions~\cite{Niklasson_canonical_pert_2015}.

\section{Relation to complex step approximations}
Let the function $f(x)$ be real-valued for real $x$ and assume
that $f$ is complex differentiable. 
Then, the complex step approximation 
\begin{equation}
  f'(x_0) = \textrm{Im}(f(x_0+ih))/h +O(h^2)
\end{equation}
can
be used to evaluate the derivative numerically \cite{complex_step_original}. 
In contrast to regular finite difference approximations, this
approximation does not suffer from subtractive cancellation
errors. Therefore, the choice of step length $h$ does not involve the
usual trade-off between truncation and cancellation errors.  As long
as underflow is avoided $h$ can be chosen arbitrarily small.  Since
the approximation is second order accurate, this means that $h$ 
typically can be chosen so that the overall accuracy is determined by the
accuracy of the function evaluation.
The observations in this section are based on
two different generalizations of the complex step approximation
above.

The first generalization is the extension of the approximation to the
evaluation of first order Fr\'{e}chet derivatives of matrix functions.
Al-Mohy and Higham~\cite{Al-MohyHigham_derivative_2009} showed that
for a matrix function $f(A)$, real-valued for real input, we have that
the Fr\'{e}chet derivative
\begin{equation} \label{eq:complex_step_frechet}
  L_f(A,E) = \textrm{Im}(f(A+ihE))/h + O(h^2)
\end{equation}
if $A,E$ are real.

The second generalization is the extension of the scalar complex step
approximation to higher order derivatives.  To this end, Lantoine et
al.\ made use of multicomplex numbers~\cite{complex_step_scalar_higher}.
The $j$-complex numbers can be defined recursively, similarly to
regular complex (1-complex) numbers but with $(j-1)$-complex numbers
used in place of the real coefficients, i.e.\ 
$\mathbb{C}^j = \{z_1+z_2i_j : z_1,z_2\in \mathbb{C}^{j-1}\}$, $j>0$
and $\mathbb{C}^0 = \mathbb{R}$. Here, $i_1,\dots, i_j$ are imaginary
units with the property $i_1^2 =\dots=i_j^2 = -1$ and $\mathbb{C}^1$
is the regular complex numbers.  We define the imaginary function
$\operatorname{Im}_{i}(z)$ of a $j$-complex number $z$ as the
$(j-1)$-complex coefficient of the imaginary unit $i$. For example,
$\operatorname{Im}_{i_1}(a + i_1 b + i_2 c + i_1i_2 d) = b + i_2 d$.
Lantoine et al.\ propose complex step approximations for higher order
partial derivatives of real functions of several variables. For
simplicity, we write here the result for a function of one variable. The
$j$th order derivative may be approximated by
\begin{equation}\label{eq:complex_step_higher}
  f^{(j)}(x_0) = \underset{i_1}{\operatorname{Im}}(\dots
  \underset{i_j}{\operatorname{Im}} (f(x_0+i_1h+\dots +
  i_jh))\dots)/h^j + O(h^2).
\end{equation}

To see the relation to the development in the previous sections and to
generalize the complex step approximations above to higher order
derivatives of matrix functions with complex input, we introduce a
block matrix representation of complex and multicomplex matrices.
A $j$-complex matrix $A+i_jB: A,B\in \mathbb{C}_{n\times n}^{j-1}$ can
be represented as a $(j-1)$-complex $2n\times 2n$ matrix
\begin{align}
  \begin{bmatrix}
    A & B \\ -B & A
  \end{bmatrix}
\end{align}
and this representation may be repeated recursively until a real
matrix is obtained with dimension $2^jn \times 2^jn$.
For example, such a block matrix representation of the $j$-complex
matrix $A_0+i_1hE_1+\dots+i_jhE_j$ is given by the recursion 
\begin{equation} \label{eq:j-complex_pure}
  \begin{aligned} 
    X_0 & = A_0, \\
    X_i & = 
    \begin{bmatrix}
      X_{i-1}                & I_{2^{i-1}} \otimes hE_i  \\
      -I_{2^{i-1}} \otimes hE_i & X_{i-1}
    \end{bmatrix}, \quad i = 1,\dots,j.
  \end{aligned}
\end{equation}
We assume in the following that $f(A)$ is analytic and real-valued for real $A$.
We show in Appendix~\ref{app:complex} that
\begin{equation}\label{eq:CS_block_Frechet_first_order_complex}
  L_f(A_0,E_1) = \frac{1}{h}\left[f(X_1)
    \right]_{1,2}
  + O(h^2)
\end{equation}
and
\begin{equation}\label{eq:CS_block_Frechet_second_order_complex}
  L^{(2)}_f(A_0,E_1,E_2) = \frac{1}{h^2}\left[f(X_2)
    \right]_{1,4}
  + O(h^2)
\end{equation}
hold with complex $A_0,E_1,E_2$.  Note that the coupled imaginary part
retrieved by $\underset{i_1}{\operatorname{Im}}(\dots
\underset{i_j}{\operatorname{Im}}(Z)\dots)$ of a $j$-complex matrix
$Z$ is given by the upper right $(1,2^j)$ submatrix of the block
representation. Therefore, the formulas
\eqref{eq:CS_block_Frechet_first_order_complex} and
\eqref{eq:CS_block_Frechet_second_order_complex} may be interpreted as
multicomplex step approximations and may be written as
\begin{equation}\label{eq:CS_Frechet_first_order_complex}
  L_f(A_0,E_1) = \underset{i_1}{\operatorname{Im}}[f(A_0+i_1hE_1)]/h+O(h^2)
\end{equation}
and
\begin{equation}\label{eq:CS_Frechet_second_order_complex}
  L^{2}_f(A_0,E_1,E_2) = \underset{i_1}{\operatorname{Im}}[\underset{i_2}{\operatorname{Im}}[f(A_0+i_1hE_1+i_2hE_2)]]/h^2+O(h^2)
\end{equation}
with $i_1$ and $i_2$ distinct from the regular imaginary unit $i$.

By Theorem~\ref{thm:derivative_as_linear_combination_of_frechet}, 
\eqref{eq:CS_block_Frechet_first_order_complex}, and
\eqref{eq:CS_block_Frechet_second_order_complex} we
have that a second order general partial derivative $\partial^\alpha
f(A(x))|_{x=\bar{x}}$, $\alpha = \beta+\gamma$, $|\beta|=|\gamma|=1$
is given by a linear combination 
$\partial^\alpha f(A(x))|_{x=\bar{x}} =
\frac{1}{h^2}[f(X_1)]_{1,4} +\frac{1}{h}[f(X_2)]_{1,2} +
O(h^2)$ where
\begin{align}
  X_1 = 
  \begin{bmatrix}
    A            & hA^{(\beta)}   & hA^{(\gamma)} & 0           \\
    -hA^{(\beta)}  & A            & 0           & hA^{(\gamma)} \\
    -hA^{(\gamma)} & 0            & A           & hA^{(\beta)}  \\
     0           & -hA^{(\gamma)} & -hA^{(\beta)} & A
  \end{bmatrix},
  \quad
  X_2 = 
  \begin{bmatrix}
    A            & hA^{(\alpha)} \\
    -hA^{(\alpha)} & A               
  \end{bmatrix}. 
\end{align}

The derivative may also be computed directly without going via the
Fr{\'e}chet derivatives. We show in Appendix~\ref{app:complex} that 
\begin{equation}\label{eq:second_order_complex_step}
  \partial^\alpha f(A(x))|_{x=\bar{x}} =
  \frac{1}{h^2}[f(X)]_{1,4} + O(h^2).
\end{equation}
where 
\begin{align}\label{eq:X_complex_step}
  X = 
  \begin{bmatrix}
    A              & hA^{(\beta)}   & hA^{(\gamma)}    & h^2A^{(\alpha)} \\
    -hA^{(\beta)}     & A            & -h^2A^{(\alpha)} & hA^{(\gamma)}   \\
    -hA^{(\gamma)}   & -h^2A^{(\alpha)} & A             & hA^{(\beta)}    \\
     h^2A^{(\alpha)} & -hA^{(\gamma)}  & -hA^{(\beta)}    & A
  \end{bmatrix}.
\end{align}  

The block matrix representation makes the relation to the block upper
triangular form introduced in Section~\ref{sec:existence} 
evident. Theorem~\ref{thm:exist_higher_deriv} tells us 
that for example
\begin{equation}\label{eq:second_order_block_upper}
  \partial^\alpha f(A(x))|_{x=\bar{x}} = [f(X)]_{1,4}
\end{equation}
with
\begin{align}\label{eq:X_upper_triangular}
  X = 
  \begin{bmatrix}
    A & A^{(\beta)} & A^{(\gamma)} & A^{(\alpha)} \\
    0 & A         & 0          & A^{(\gamma)} \\
    0 & 0         & A          & A^{(\beta)}  \\
    0 & 0         & 0          & A
  \end{bmatrix}
\end{align}  
which can be compared with \eqref{eq:X_complex_step}.  The essential
difference lies in the nonzero lower left block in the recursive
construction of the block matrix.  For example, the representation in
\eqref{eq:j-complex_pure} can be compared to the corresponding
recursive definition in~\eqref{eq:recursive_frechet}.  Through the
lower left block higher order terms in $h$ may enter the calculation
and contaminate the result, leading to the $O(h^2)$ error of the complex
step approximation.

For higher order Fr{\'e}chet derivatives, hybrid methods have been
proposed where the complex step approximation is used only in one of
the variables of the Fr{\'e}chet
derivative~\cite{complex_step_higher_frechet_2021,
  werner_arxiv_2022}. In the block formalism developed here, block
forms for such methods are given by the use of the complex step
approximation only in the last level of the recursive construction of
$X$.
Such a hybrid approach may in principle also be used for the general partial
derivative and the block form for such a method would for the second
order case be given by 
\begin{align}\label{eq:X_hybrid}
  X = 
  \begin{bmatrix}
    A           &  A^{(\beta)}    &  hA^{(\gamma)} & hA^{(\alpha)} \\
    0           &  A            &  0           & hA^{(\gamma)} \\
   -hA^{(\gamma)} & -hA^{(\alpha)}  &  A           &  A^{(\beta)}  \\
    0           & -hA^{(\gamma)}  &  0           & A
  \end{bmatrix}
\end{align}  
and
\begin{equation}\label{eq:second_order_hybrid}
  \partial^\alpha f(A(x))|_{x=\bar{x}} = \frac{1}{h}[f(X)]_{1,4} + O(h^2).
\end{equation}

In all three block matrix forms \eqref{eq:X_complex_step},
\eqref{eq:X_upper_triangular}, and \eqref{eq:X_hybrid}, a second order
Fr{\'e}chet derivative $L_f^{(2)}(A,E_1,E_2)$ can be represented by
replacing $A^{(\beta)}$, $A^{(\gamma)}$, and $A^{(\alpha)}$ with
$E_1$, $E_2$, and 0, respectively.

\section{Numerical experiments}
We perform a few numerical experiments with the purpose of
highlighting the relationships between the methods discussed in the
previous section.
We use an example from~\cite{Al-MohyHigham_derivative_2009}, where the
complex step approximation and the algorithm $\cos(A) =
(e^{iA}+e^{-iA})/2$ is used to evaluate $L_{\textrm{cos}}(A,E)$. Using
this algorithm in \eqref{eq:complex_step_frechet} gives the
approximation
\begin{equation} \label{eq:cos_der_complex_step_regular}
  L_{\textrm{cos}}(A,E) \approx \textrm{Im}((e^{iA-hE}+e^{-iA+hE}))/(2h).
\end{equation}
We refer to this as the regular complex step approximation.  
If we instead use the block form of
\eqref{eq:CS_block_Frechet_first_order_complex}
we get
\begin{equation} \label{eq:cos_der_complex_step_block_matrix}
  L_{\textrm{cos}}(A,E) \approx
  \frac{1}{2h}\left[
    \exp\left(i
    \begin{bmatrix}
      A & hE \\ -hE & A
    \end{bmatrix}
    \right)
    +
    \exp\left(-i
    \begin{bmatrix}
      A & hE \\ -hE & A
    \end{bmatrix}
    \right)
    \right]_{1,2}
\end{equation}
which we refer to as the block matrix complex step approximation. We
also compare with the standard central finite difference
\begin{equation}\label{eq:cos_der_finite_difference}
  L_{\textrm{cos}}(A,E) \approx (\cos(A+hE)-\cos(A-hE))/(2h).
\end{equation}
Finally, we compare with the block upper triangular form provided by
Theorem~\ref{thm:exist_frechet}, which reduces to Mathias' result in
this case, that is
\begin{equation}\label{eq:cos_der_block_upper}
  L_{\textrm{cos}}(A,E) \approx
  \frac{1}{2h}\left[
    \exp\left(i
    \begin{bmatrix}
      A & hE \\ 0 & A
    \end{bmatrix}
    \right)
    +
    \exp\left(-i
    \begin{bmatrix}
      A & hE \\ 0 & A
    \end{bmatrix}
    \right)
    \right]_{1,2}.
\end{equation}
We include here a scaling of $E$ to see if this affects the accuracy
in this case, as discussed in~\cite{Al-MohyHigham_derivative_2009,
  almohy_exp_frechet}.  Note that such scaling may be used in any
algorithm for the Fr{\'e}chet derivative since $L_{f}(A,E)$ is linear
in $E$.  Throughout this section, we use {\tt expm} in Matlab R2022a to
evaluate all matrix exponentials~\cite{higham_expm_2005,
  al_mohy_higham_expm_2010}.

We use first scalar input $A=E=1$ as in
\cite{Al-MohyHigham_derivative_2009}, with results for the four
alternatives above in Panel~(a) of Figure~\ref{fig:first_order}. In
this case, the regular complex step approximation suffers from
subtractive cancellation errors, similarly to the finite difference
approximation. This is due to the algorithm using complex arithmetic,
as discussed in \cite{Al-MohyHigham_derivative_2009}. This problem
disappears when the block matrix formulation is used. The block form
can be seen as using a separate imaginary unit in the complex step.
In exact arithmetics, the two approaches
should give the same result, but numerically the block matrix
formulation is preferred.  We repeat the experiment in Panel~(b) with
complex scalar input with the real and imaginary parts of $A$ and $E$
each drawn from a uniform distribution in $[0,\, 1]$.  As expected,
the regular complex step approximation does not work in this case
since it is only able to produce real output.  However, with the
distinction between the imaginary units provided by the block form,
the complex step approximation works also with complex input and
output.

\begin{figure}
  \begin{center}
    \subfigure[Real\label{fig:cos_der_real}]{
      \includegraphics[width=0.45\textwidth]{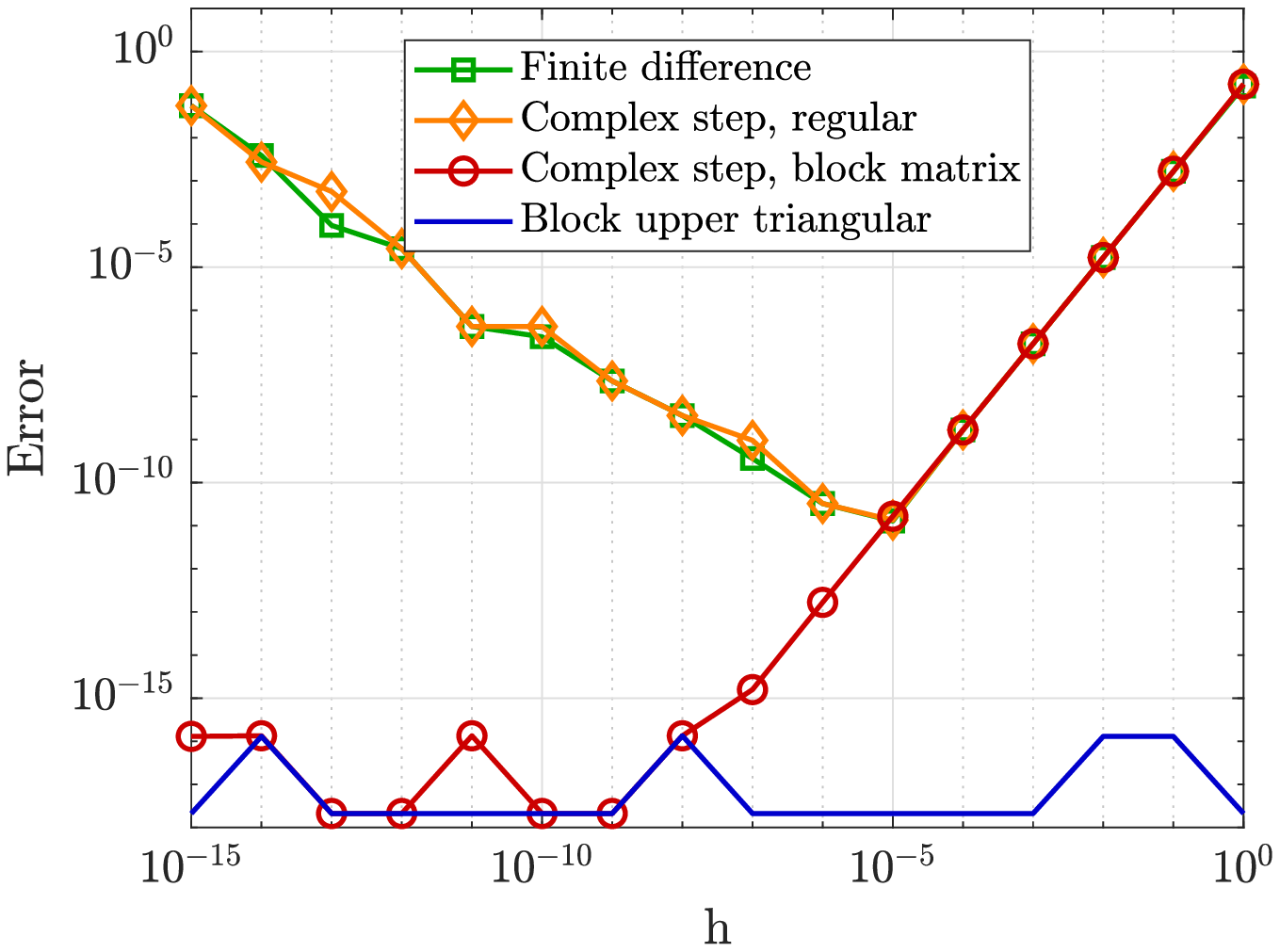}
    }
    \subfigure[Complex\label{fig:cos_der_complex}]{
      \includegraphics[width=0.45\textwidth]{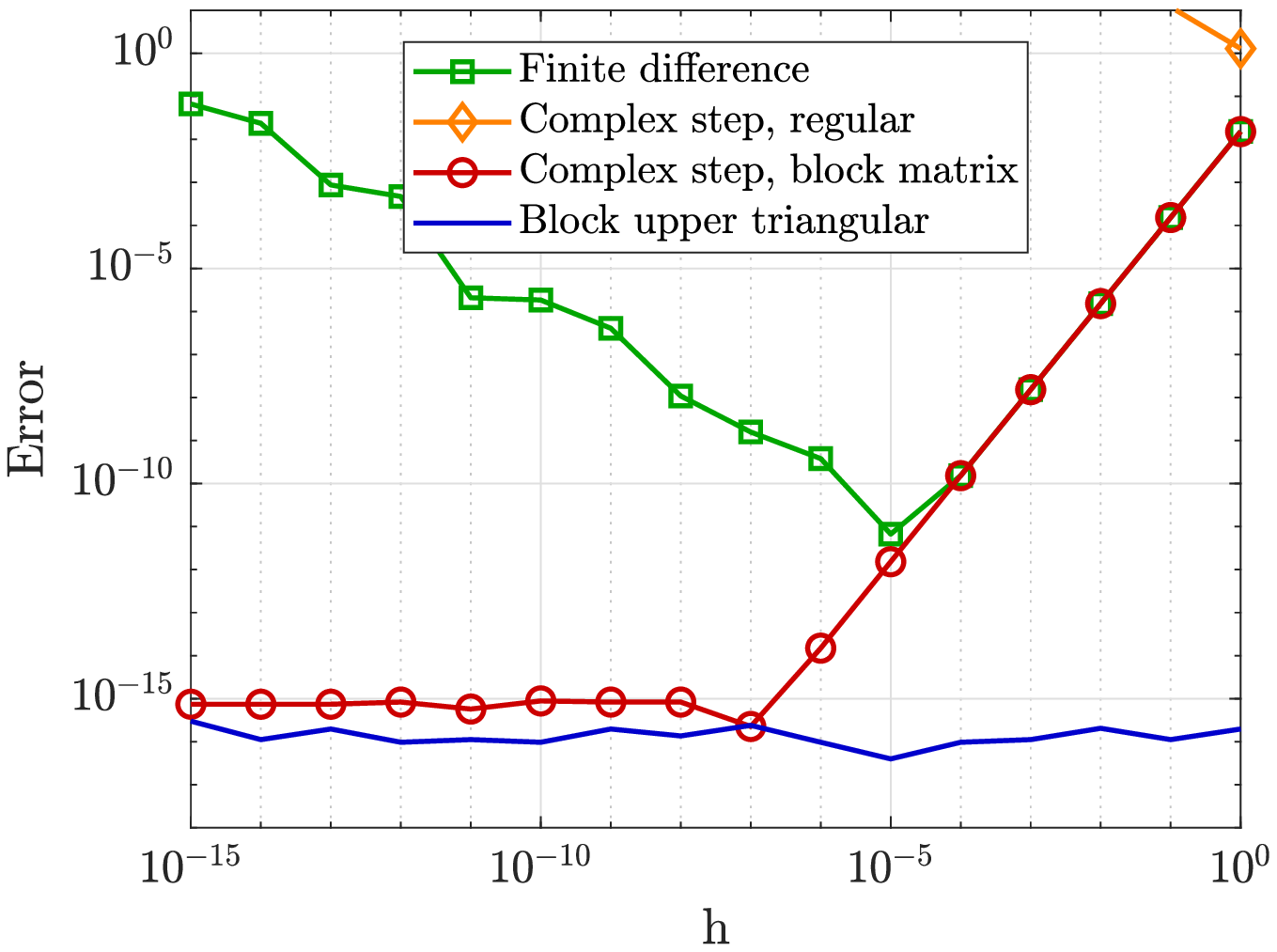}
    }
  \end{center}
  \caption{Relative errors in computations of $L_{\textrm{cos}}(A,E)$
    using a standard central finite difference
    formula~\eqref{eq:cos_der_finite_difference}, the regular
    \eqref{eq:cos_der_complex_step_regular} and block matrix
    \eqref{eq:cos_der_complex_step_block_matrix} complex step
    approximations, and the block upper triangular form
    \eqref{eq:cos_der_block_upper}.  The error is given as a function
    of step length $h$. Panel~\subref{fig:cos_der_real}: errors with
    real scalar input $A=E=1$. Panel~\subref{fig:cos_der_complex}: errors with
    random complex scalar input.
    \label{fig:first_order} }
\end{figure}  

We consider also a general partial derivative
\begin{equation}\label{eq:partial_second_derivative_cos}
  \frac{\partial^2}{\partial x \partial y} \cos(A(x,y))|_{x=y=0}.
\end{equation}
We let the derivatives $A(0,0)$, $A'_x(0,0)$, $A'_y(0,0)$, and
$A''_{xy}(0,0)$ needed to characterize $A(x,y)$ be random complex $3
\times 3$ matrices with elements whose real and imaginary parts are
each drawn from a uniform distribution in $[-0.5,\, 0.5]$.

We compare the second order complex step approximation
\eqref{eq:second_order_complex_step}, the second order block upper
triangular form \eqref{eq:second_order_block_upper}, the hybrid
approach \eqref{eq:second_order_hybrid} and the finite difference
approximation
\begin{multline}\label{eq:cos_der_second_order_finite_difference}
  (\cos(A+hA'_x+hA'_y+h^2A''_{xy})
  -\cos(A+hA'_x-hA'_y-h^2A''_{xy}) \\
  -\cos(A-hA'_x+hA'_y-h^2A''_{xy})
  +\cos(A-hA'_x-hA'_y+h^2A''_{xy}))/(4h^2)
\end{multline}
where the matrix cosine is evaluated using diagonalization.
Figure~\ref{fig:matrix_test} shows the relative error in the spectral
norm approximated by
$\|\widetilde{X}-X_\textrm{ref}\|_2/\|X_\textrm{ref}\|_2$ where
$\widetilde{X}$ is the approximation and $X_\textrm{ref}$ is an
accurate reference solution computed using the finite difference
approximation evaluated in high precision using Matlab's symbolic
toolbox with a step size of $h=10^{-30}$.
We have also tried the approach to compute the partial derivative as a
sum of Fr{\'e}chet derivatives, which in each case gives errors at the
same level as the combined approach whose errors are shown in the
figure.

\begin{figure}
  \begin{center}
    \includegraphics[width=0.45\textwidth]{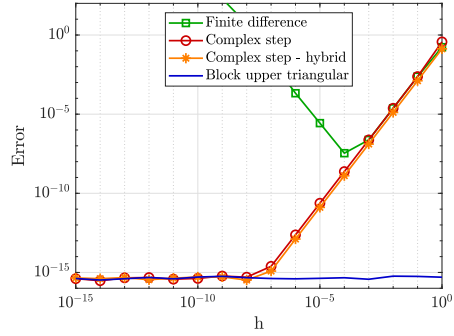}
  \end{center}
  \caption{Relative errors in computations of the partial derivative
    in~\eqref{eq:partial_second_derivative_cos} using the finite
    difference
    approximation~\eqref{eq:cos_der_second_order_finite_difference},
    the complex step approximation given
    by~\eqref{eq:second_order_complex_step}, the block upper
    triangular form~\eqref{eq:second_order_block_upper}, and the
    hybrid approach of~\eqref{eq:second_order_hybrid}.  The
    computations were carried out with $A(x,y)$
    in~\eqref{eq:partial_second_derivative_cos} being a $3\times 3$
    complex matrix with complex valued derivatives at $x=y=0$. See the
    text for details.
    \label{fig:matrix_test}}
\end{figure}  

\section{Discussion}
In this article, we develop a block upper triangular formulation for
higher order derivatives of matrix functions. This form was shown to
generalize classical formulas of quantum perturbation theory. More
recent developments within quantum perturbation theory can also be
derived from
the block upper triangular form, such as
methods based on recursive Fermi operator expansions~\cite{weber_2005,
  JFinkelstein22}. We expect the formulation to be useful for further
development and analysis of methods also in this class.

The block upper triangular structure is preserved under matrix
function evaluation and typically the repeated blocks below the first
block row do not need to be involved in the computation. In
\eqref{eq:daleckii_krein_generalized} a diagonalization of
$A(\bar{x})$ provides a Schur form for the whole block matrix, and
only the desired matrix blocks of the result need to be computed. In
iterative methods, such as recursive Fermi operator
expansions~\cite{JFinkelstein22}, only the first block row needs to be
computed in each iteration. This has been discussed previously for
Newton iterations and first order Fr{\'e}chet
derivatives~\cite{Mathias_derivative_1996, almohy_exp_frechet}.

We introduced a block matrix formulation of first and higher order complex
step approximations and show how they are related to the block upper
triangular form. The essential difference is that the complex step
approximations include additional terms in the block matrix.
Those terms are responsible for the truncation errors in the complex
step approximation. Often, the additional terms also mean that more
computations are required. The main advantage with the complex step
approach is that one may use an existing complex number
representation.  In the case of higher derivatives, a generic complex
type may be used to recursively construct multicomplex types.
Such types can also be used to handle complex input.

The relationship between the block matrix forms is analogous to that
of complex step approximations and automatic differentiation,
discussed previously for example in~\cite{complex_step_auto_diff}. 
Automatic differentiation makes use of dual numbers $a+\varepsilon b$,
where $a$ and $b$ are real numbers and $\varepsilon$ satisfies
$\varepsilon^2 = 0$. A dual number may be represented by a matrix
$\begin{bmatrix}a & b \\ 0 & a\end{bmatrix}$, which may be compared to
  the block upper triangular matrix formulation considered in the
  present work.  Similarly, a complex scalar number $a+ib$ may be
  represented in matrix form as $\begin{bmatrix}a & b \\ -b &
    a\end{bmatrix}$.
  
\appendix
\section{Relation to perturbation theory}\label{app:perturbation_theory}
We show here that several theories of perturbations of Hermitian
matrices can be seen as special cases or derived from the generalized
Dalecki{\u\i}-Kre{\u\i}n formulas. 

We consider derivatives on the form 
\begin{equation}
  P^{(\alpha)} =
  \left.\partial^\alpha f(H(\varepsilon))\right|_{\varepsilon=0} 
\end{equation}
where $H(\varepsilon)$ is Hermitian and $f(x)$ is a step function
defined for scalars as
\begin{equation}
  f(x) =
  \begin{cases}
    1 & \textrm{if } x < \mu, \\
    0 & \textrm{otherwise.}
  \end{cases}
\end{equation}
As previously we use the shorthand notation $H^{(\alpha)} =
\partial^\alpha H(\varepsilon)|_{\varepsilon = 0}$, and $H^{(0)} =
H(0)$.
We assume that $H^{(0)}$ has no eigenvalues in a neighborhood of $\mu$
so that $f(x)$ and its derivatives are defined on the spectrum of
$H(\varepsilon)$ for sufficiently small values of $\varepsilon$.
Let $\lambda_1 \leq \lambda_2 \leq \dots \leq \lambda_n$ be the
ordered eigenvalues of $H^{(0)}$.
We disregard trivial cases when $\mu$ is not within the
closure of the eigenspectrum. Then, there is an $n_e$ such that 
$1 \leq n_e < n$ and $\lambda_{n_e} < \mu < \lambda_{n_e+1}$. 
In the following, we let $U=Q^*H^{(0)}Q$ be an eigendecomposition of
$H^{(0)}$ such that $U=\diag([\lambda_1, \dots, \lambda_n])$.

In electronic structure theory, the invariant subspace of $H^{(0)}$ spanned by eigenvectors corresponding to eigenvalues smaller than $\mu$ is sometimes referred to as the occupied subspace. $P^{(0)} = f(H^{(0)})$, known as the density matrix, is the matrix for orthogonal projection onto the occupied subspace. Similarly, $I-P^{(0)}$ is the projection onto the so-called virtual subspace spanned by the remaining eigenvectors.

The divided differences
$f[\lambda_i,\lambda_j], \ i,j = 1,\dots, n$ are nonzero for pairs of
eigenvalues on opposite sides of $\mu$:
\begin{equation}
  f[\lambda_i,\lambda_j] =
  \begin{cases}
    -\frac{1}{|\lambda_i-\lambda_j|} & \textrm{if } (\lambda_i-\mu)(\lambda_j-\mu) < 0, \\
    0 & \textrm{otherwise.}
  \end{cases} 
\end{equation}

We define the corresponding $n \times n$ Loewner matrix with entries $D_{ij} =
f[\lambda_i,\lambda_j]$ and let $V^{(\alpha)} = Q^*P^{(\alpha)}Q$.
In the first order case ($|\alpha| = 1$) we get from
\eqref{eq:daleckii_krein} immediately that
\begin{equation}\label{eq:first_order_heavi}
  V^{(\alpha)} = Q^*(\partial^\alpha f(H(\varepsilon))|_{\varepsilon=0})Q = D \circ U^{(\alpha)} 
\end{equation}
which, due to the structure of $D$, is a block $2\times 2$ matrix with zero diagonal blocks.
The off-diagonal blocks describe transitions between the occupied and
virtual subspaces of $H^{(0)}$.
This is the first order correction of the density matrix
perturbation theory by McWeeny and Diercksen~\cite{mcweeny_diercksen_1966}.

For higher orders we need the second order divided differences
\begin{equation}\label{eq:heavi_divided_difference_second_order}
  f[\lambda_i,\lambda_j, \lambda_k] =
  \begin{cases}
    0 & \textrm{if } \lambda_i < \mu, \lambda_j < \mu, \lambda_k < \mu, \\
    0 & 
    \textrm{if }
    \lambda_i > \mu, \lambda_j > \mu, \lambda_k > \mu, \\
    -\frac{1}{|\lambda_k-\lambda_j||\lambda_i-\lambda_k|} &
    \textrm{if } \lambda_i < \mu, \lambda_j < \mu, \lambda_k > \mu, \\
    \frac{1}{|\lambda_k-\lambda_j||\lambda_i-\lambda_k|} &
    \textrm{if } \lambda_i > \mu, \lambda_j > \mu, \lambda_k < \mu.
  \end{cases}
\end{equation}
The other cases are given by permutations of the arguments of
$f[\lambda_i,\lambda_j, \lambda_k]$ which is invariant to their order.
For example, the case with $\lambda_i<\mu$, $\lambda_j>\mu$,
$\lambda_k<\mu$ is covered by the third case of
\eqref{eq:heavi_divided_difference_second_order} since
$f[\lambda_i,\lambda_j, \lambda_k] = f[\lambda_i,\lambda_k,
  \lambda_j]$.

The second order formula ($|\alpha| = 2$) is a specialization of
\eqref{eq:second_order_explicit} using the second order divided
differences in \eqref{eq:heavi_divided_difference_second_order} and
the formula for first order \eqref{eq:first_order_heavi}.
As an example, we look at one of the terms of
\eqref{eq:second_order_explicit} for the case when $i \leq n_e$ and
$j>n_e$:
\begin{multline} 
  \sum_{k=1}^n U_{i,k}^{(\beta)}U_{k,j}^{(\gamma)} f[\lambda_i,\lambda_k,\lambda_j] = \\
  -\frac{1}{|\lambda_i-\lambda_j|}\sum_{k=1}^{n_e} U_{i,k}^{(\beta)}\underbrace{U_{k,j}^{(\gamma)}\frac{1}{|\lambda_j-\lambda_k|}}_{-V_{k,j}^{(\gamma)}} +
  \frac{1}{|\lambda_i-\lambda_j|}\sum_{k=n_e+1}^{n} \underbrace{U_{i,k}^{(\beta)} \frac{1}{|\lambda_i-\lambda_k|}}_{-V_{i,k}^{(\beta)}} U_{k,j}^{(\gamma)} 
\end{multline}
which can be written in matrix form as
\begin{equation}
  D\circ \left(
  \begin{bmatrix}
    I & 0 \\ 0 & 0
  \end{bmatrix}
  V^{(\beta)}
  U^{(\gamma)}
  \begin{bmatrix}
    0 & 0 \\ 0 & I
  \end{bmatrix}  
  - 
  \begin{bmatrix}
    I & 0 \\ 0 & 0
  \end{bmatrix}
  U^{(\beta)}
  V^{(\gamma)}
  \begin{bmatrix}
    0 & 0 \\ 0 & I
  \end{bmatrix}  
  \right).
\end{equation}
Here, the submatrices of the block $2\times 2$ matrices have
dimensions conforming to the division of the eigenspectrum in its
occupied and virtual parts. Following the same procedure for the other
terms we get that
\begin{multline}\label{eq:second_order_heavi}
  V^{(\alpha)} = Q^*(\partial^\alpha f(H(\varepsilon))|_{\varepsilon=0})Q = \\
  \begin{aligned}[b]
    & D \circ U^{(\alpha)} +
    \\
    &
    D\circ \left(
    \begin{bmatrix}
      I & 0 \\ 0 & 0
    \end{bmatrix}
    V^{(\beta)}
    U^{(\gamma)}
    \begin{bmatrix}
      0 & 0 \\ 0 & I
    \end{bmatrix}  
    - 
    \begin{bmatrix}
      I & 0 \\ 0 & 0
    \end{bmatrix}
    U^{(\beta)}
    V^{(\gamma)}
    \begin{bmatrix}
      0 & 0 \\ 0 & I
    \end{bmatrix}  
    \right) +
    \\
    &
    D\circ \left(
    \begin{bmatrix}
      0 & 0 \\ 0 & I
    \end{bmatrix}
    U^{(\beta)}
    V^{(\gamma)}
    \begin{bmatrix}
      I & 0 \\ 0 & 0
    \end{bmatrix}  
    - 
    \begin{bmatrix}
      0 & 0 \\ 0 & I
    \end{bmatrix}
    V^{(\beta)}
    U^{(\gamma)}
    \begin{bmatrix}
      I & 0 \\ 0 & 0
    \end{bmatrix}  
    \right) +
    \\
    &
    \begin{bmatrix}
      0 & 0 \\ 0 & I
    \end{bmatrix}
    V^{(\beta)}
    V^{(\gamma)}
    \begin{bmatrix}
      0 & 0 \\ 0 & I
    \end{bmatrix}
    -
    \begin{bmatrix}
      I & 0 \\ 0 & 0
    \end{bmatrix}
    V^{(\beta)}
    V^{(\gamma)}
    \begin{bmatrix}
      I & 0 \\ 0 & 0
    \end{bmatrix} +
    \\
    &
    D\circ \left(
    \begin{bmatrix}
      I & 0 \\ 0 & 0
    \end{bmatrix}
    V^{(\gamma)}
    U^{(\beta)}
    \begin{bmatrix}
      0 & 0 \\ 0 & I
    \end{bmatrix}  
    - 
    \begin{bmatrix}
      I & 0 \\ 0 & 0
    \end{bmatrix}
    U^{(\gamma)}
    V^{(\beta)}
    \begin{bmatrix}
      0 & 0 \\ 0 & I
    \end{bmatrix}  
    \right) +
    \\
    &
    D\circ \left(
    \begin{bmatrix}
      0 & 0 \\ 0 & I
    \end{bmatrix}
    U^{(\gamma)}
    V^{(\beta)}
    \begin{bmatrix}
      I & 0 \\ 0 & 0
    \end{bmatrix}  
    - 
    \begin{bmatrix}
      0 & 0 \\ 0 & I
    \end{bmatrix}
    V^{(\gamma)}
    U^{(\beta)}
    \begin{bmatrix}
      I & 0 \\ 0 & 0
    \end{bmatrix}  
    \right) +
    \\
    &
    \begin{bmatrix}
      0 & 0 \\ 0 & I
    \end{bmatrix}
    V^{(\gamma)}
    V^{(\beta)}
    \begin{bmatrix}
      0 & 0 \\ 0 & I
    \end{bmatrix}
    -
    \begin{bmatrix}
      I & 0 \\ 0 & 0
    \end{bmatrix}
    V^{(\gamma)}
    V^{(\beta)}
    \begin{bmatrix}
      I & 0 \\ 0 & 0
    \end{bmatrix}
  \end{aligned}
\end{multline}


We split $Q = [Q_{\textrm{occ}}\ Q_{\textrm{vir}}]$ into rectangular
matrices $Q_{\textrm{occ}}$ and $Q_{\textrm{vir}}$ containing the
eigenvectors of $H^{(0)}$ that span the occupied and virtual
subspaces, respectively.
Using also the relations
$U^{(\alpha)} = Q^*H^{(\alpha)}Q$ and $V^{(\alpha)} =
Q^*P^{(\alpha)}Q$, the second order derivative in
\eqref{eq:second_order_heavi} can be written as
\begin{multline}
  Q^*(\partial^\alpha f(H(\varepsilon))|_{\varepsilon=0})Q = \\
  D \circ \left(
  \begin{bmatrix}
    0 \qquad Q_{\textrm{occ}}^*(H^{(\alpha)}+
    P^{(\beta)}H^{(\gamma)}-H^{(\beta)}P^{(\gamma)}+P^{(\gamma)}H^{(\beta)}-H^{(\gamma)}P^{(\beta)} ) Q_{\textrm{vir}} \\
    Q_{\textrm{vir}}^*(H^{(\alpha)}+
    H^{(\beta)}P^{(\gamma)}-P^{(\beta)}H^{(\gamma)}+H^{(\gamma)}P^{(\beta)}-P^{(\gamma)}H^{(\beta)} ) Q_{\textrm{occ}} \qquad 0
  \end{bmatrix}
  \right) + \\
  \begin{bmatrix}
    -Q_{\textrm{occ}}^*(P^{(\beta)}P^{(\gamma)}+P^{(\gamma)}P^{(\beta)})Q_{\textrm{occ}} & 0 \\
    0 & Q_{\textrm{vir}}^*(P^{(\beta)}P^{(\gamma)}+P^{(\gamma)}P^{(\beta)})Q_{\textrm{vir}}
  \end{bmatrix}.
\end{multline}
For non-mixed derivatives ($\beta = \gamma$) this is basically the
second order correction of the density matrix perturbation theory by
McWeeny and Diercksen~\cite{mcweeny_diercksen_1966}.

The formulas in \eqref{eq:first_order_heavi} and
\eqref{eq:second_order_heavi} can also be used to derive classical formulas for
perturbed eigenvectors of $H(\varepsilon)$.  We consider
a first order perturbation $H(\varepsilon) = H^{(0)}+\varepsilon H^{(1)}$ and let
$n_e = 1$ and $Q=[q_1\, q_2 \dots q_n]$. The ground state
eigenvector of $H(\varepsilon)$ can be written $q(\varepsilon) =
P(\varepsilon) x$ where $x$ is any vector non-orthogonal to the
$\varepsilon$-dependent occupied subspace.
We note that the finite gap around $\mu$ implies that the ground state
eigenvalue is distinct when $n_e=1$.
Therefore, the ground state eigenvector is continuous at $\varepsilon=0$ and 
we may for sufficiently small values of $\varepsilon$ choose $x = q_1$.
We use \eqref{eq:first_order_heavi} for the first order derivative of
$P(\varepsilon)$ and get the familiar first order correction of the eigenvector
\begin{align}
  q^{(1)} & = \partial^{(1)} q(\varepsilon)|_{\varepsilon = 0} 
   = \partial^{(1)} P(\varepsilon)q_1|_{\varepsilon = 0} \\
  & = Q(D\circ(Q^*H^{(1)}Q))Q^*q_1 \\
  & = -\sum_{i=2}^n \frac{q_i^*H^{(1)}q_1}{|\lambda_1-\lambda_i|}q_i.
\end{align}

We use \eqref{eq:second_order_heavi} for the second order derivative
of $P(\varepsilon)$ and use \eqref{eq:first_order_heavi} for the first
order terms in \eqref{eq:second_order_heavi}.  Since
$\begin{bmatrix}0&0\\0&I\end{bmatrix} Q^*q_1$ is zero, half of the
  terms vanish immediately and due to the simple case with a non-mixed
  derivative the remaining terms are repeated in pairs. We get the familiar
  second order correction of the eigenvector
\begin{align}
  q^{(2)} = & \partial^{(2)} q(\varepsilon)|_{\varepsilon = 0} 
  =  \partial^{(2)} P(\varepsilon)q_1|_{\varepsilon = 0} \\
  = &
  \begin{aligned}[t]
    Q\bigg(& 2D\circ \bigg(
    \begin{bmatrix}
      0 & 0 \\ 0 & I
    \end{bmatrix}
    Q^*H^{(1)}Q (D\circ (Q^*H^{(1)}Q))
    \begin{bmatrix}
      1 & 0 \\ 0 & 0
    \end{bmatrix}
    \bigg)-
    \\
    &
    2 D\circ \bigg(
    \begin{bmatrix}
      0 & 0 \\ 0 & I
    \end{bmatrix}
    (D\circ (Q^*H^{(1)}Q)) Q^*H^{(1)}Q 
    \begin{bmatrix}
      1 & 0 \\ 0 & 0
    \end{bmatrix}
    \bigg)-
    \\
    &
    2 D\circ \bigg(
    \begin{bmatrix}
      1 & 0 \\ 0 & 0
    \end{bmatrix}
    (D\circ (Q^*H^{(1)}Q))^2
    \begin{bmatrix}
      1 & 0 \\ 0 & 0
    \end{bmatrix}
    \bigg)   
    \bigg)Q^*q_1
    \end{aligned}
  \\
  = &
  2\sum_{j=2}^n \sum_{i=2}^n
  \frac{q_j^*H^{(1)}q_i q_i^*H^{(1)}q_1}{|\lambda_1-\lambda_i||\lambda_1-\lambda_j|}
   q_j-
  \\
  &
  2\sum_{j=2}^n \frac{q_j^*H^{(1)}q_1q_1^*H^{(1)}q_1}{(\lambda_1-\lambda_j)^2}q_j-
  \\
  &
  2\sum_{i=2}^n \frac{(q_1^*H^{(1)}q_i)^2}{(\lambda_1-\lambda_i)^2}q_1.
\end{align}

\section{Block forms of complex step approximations}\label{app:complex}

\begin{lemma}\label{lem:jordan_block_complex}
  Let $X = VJV^{-1}$ be a Jordan decomposition of $X$ and let
  $A=\mathrm{Re}\, X$ and $B=\mathrm{Im}\, X$. Then,
  \begin{align}
    \begin{bmatrix}
      A & B \\
      -B & A
    \end{bmatrix}
    = 
    \begin{bmatrix}
      V  & \overline{V} \\
      iV & \overline{iV}
    \end{bmatrix}
    \begin{bmatrix}
      J & 0 \\
      0 & \overline{J}
    \end{bmatrix}
    \begin{bmatrix}
      V  & \overline{V} \\
      iV & \overline{iV}
    \end{bmatrix}^{-1}
  \end{align}
  is a Jordan decomposition of
  $\begin{bmatrix}
    A & B \\
    -B & A
  \end{bmatrix}$
  where $\overline{V}$ is the elementwise complex conjugate of $V$.  
\end{lemma}
\begin{proof}
  We have that
  \begin{align}
    \begin{bmatrix}
      A & B \\
      -B & A
    \end{bmatrix}
    \begin{bmatrix}
      V & \overline{V} \\
      iV & \overline{iV}
    \end{bmatrix}
    & =
    \begin{bmatrix}
      (A+iB)V & (A-iB)\overline{V} \\
      (A+iB)iV & (A-iB)\overline{iV}
    \end{bmatrix}
    =
    \begin{bmatrix}
      V  & \overline{V} \\
      iV & \overline{iV}
    \end{bmatrix}
    \begin{bmatrix}
      J & 0 \\
      0 & \overline{J}
    \end{bmatrix}
  \end{align}
  since the original Jordan decomposition of $X$ implies
  \begin{align}
    (A+iB)V & = VJ, \\
    (A-iB)\overline{V} & = \overline{V}\overline{J}.
  \end{align}
  The matrix $
    \begin{bmatrix}
      V  & \overline{V} \\
      iV & \overline{iV}
    \end{bmatrix}
    $
    is nonsingular which can be seen for example by noting that its Schur complement $\overline{iV}-iVV^{-1}\overline{V} = -2i\overline{V}$ is nonsingular~\cite{HornJohnson_matrix_analysis_2013}.
\end{proof}

\begin{theorem}\label{thm:matfun_complex}
  Let $X = VJV^{-1}$ be a Jordan decomposition of $X = A+iB$, where
  $A$ and $B$ are real matrices. Let $f$ be a function such that
  $\overline{f(J)} = f(\overline{J})$.
  Then,
  \begin{align}
    f
    \begin{pmatrix}
      A & B \\
      -B & A
    \end{pmatrix}
    & =
    \begin{bmatrix}
      \mathrm{Re}[f(X)]  & \mathrm{Im}[f(X)] \\
      -\mathrm{Im}[f(X)] & \mathrm{Re}[f(X)]
    \end{bmatrix}    
  \end{align}
\end{theorem}
\begin{proof}  
  Let $WJ_fW^{-1}$ be a Jordan decomposition of $f(J)$. Then we have
  that $f(X) = Vf(J)V^{-1} = V_fJ_fV_f^{-1}$ with $V_f = VW$.
  Furthermore, we have that $f(\overline{J}) = \overline{f(J)} =
  \overline{W}\overline{J_f}\overline{W}^{-1}$ and we get the desired
  result as follows:  
  \begin{align}
    f
    \begin{pmatrix}
      A & B \\
      -B & A
    \end{pmatrix}
    &= 
    \begin{bmatrix}
      V & \overline{V} \\
      iV & \overline{iV}
    \end{bmatrix}
    \begin{bmatrix}
      f(J) & 0 \\
      0 & f(\overline{J})
    \end{bmatrix}
    \begin{bmatrix}
      V & \overline{V} \\
      iV & \overline{iV}
    \end{bmatrix}^{-1}
    \label{eq:real_f_first}
    \\
    &=
    \begin{bmatrix}
      V & \overline{V} \\
      iV & \overline{iV}
    \end{bmatrix}
    \begin{bmatrix}
      W & 0 \\
      0 & \overline{W}
    \end{bmatrix}    
    \begin{bmatrix}
      J_f & 0 \\
      0 & \overline{J_f}
    \end{bmatrix}
    \begin{bmatrix}
      W & 0 \\
      0 & \overline{W}
    \end{bmatrix}^{-1}    
    \begin{bmatrix}
      V & \overline{V} \\
      iV & \overline{iV}
    \end{bmatrix}^{-1}
    \\
    &=
    \begin{bmatrix}
      V_f & \overline{V_f} \\
      iV_f & \overline{iV_f}
    \end{bmatrix}
    \begin{bmatrix}
      J_f & 0 \\
      0 & \overline{J_f}
    \end{bmatrix}
    \begin{bmatrix}
      V_f & \overline{V_f} \\
      iV_f & \overline{iV_f}
    \end{bmatrix}^{-1}
    \\
    &=
    \begin{bmatrix}
      \mathrm{Re}[f(X)]  & \mathrm{Im}[f(X)] \\
      -\mathrm{Im}[f(X)] & \mathrm{Re}[f(X)]
    \end{bmatrix}    \label{eq:real_f_last}
  \end{align}
   Lemma~\ref{lem:jordan_block_complex} was used in
   \eqref{eq:real_f_first} and \eqref{eq:real_f_last}.
\end{proof}
See \cite{mackey_2005, Higham_functions_of_matrices_2008} for useful
conditions under which $f(\overline{J}) = \overline{f(J)}$ holds.
In particular, this property characterizes whether $f(A)$ is
real-valued for real $A$, which is assumed hereinafter.
We will also assume that $f$ is analytic.
By Theorem~\ref{thm:matfun_complex} and the definition of the
Fr{\'e}chet derivative, we have that
\begin{align}\label{eq:frechet_complex_block_form}
  L_f\left(
    \begin{bmatrix}
      \phantom{-}A \ B \\
      -B \ A
    \end{bmatrix},    
    \begin{bmatrix}
      \phantom{-}E \ F \\
      -F \ E
    \end{bmatrix}
    \right)
    =
    \begin{bmatrix}
      \mathrm{Re}[L_f(A+iB,E+iF)]  & \mathrm{Im}[L_f(A+iB,E+iF)] \\
      -\mathrm{Im}[L_f(A+iB,E+iF)] & \mathrm{Re}[L_f(A+iB,E+iF)]
    \end{bmatrix}
\end{align}
Since the two matrices on the left hand side are both real, we can use
the complex step approximation by Al-Mohy and
Higham~\cite{Al-MohyHigham_derivative_2009} and get that
\begin{align}\label{eq:error_in_complex_step_for_complex_matrices}
  L_f\left(
    \begin{bmatrix}
      A & B \\
      -B & A
    \end{bmatrix},    
    \begin{bmatrix}
      E & F \\
      -F & E
    \end{bmatrix}
    \right)
    & = \frac{1}{h}\textrm{Im} f
    \begin{pmatrix}
      A+ihE & B+ihF \\
      -(B+ihF) & A+ihE
    \end{pmatrix}    
    \\ & + \frac{h^2}{6} L_f^{(3)}\left(
    \begin{bmatrix}
      A & B \\
      -B & A
    \end{bmatrix},    
    \begin{bmatrix}
      E & F \\
      -F & E
    \end{bmatrix}
    \right)
    + O(h^4) \nonumber
\end{align}
In combination with \eqref{eq:frechet_complex_block_form} this gives
us a complex step approximation for derivatives of functions with
complex matrix input, in general $L_f(A+iB,E+iF) = X_{11}+iX_{12} + O(h^2)$, with $X=\frac{1}{h}\textrm{Im} f
\begin{pmatrix}
  A+ihE & B+ihF \\
  -(B+ihF) & A+ihE
\end{pmatrix}$.
However, this approximation is not quite on the form we are looking for.
In particular, note that the real and imaginary parts of the input
matrices enter the matrix in different blocks.
We use Theorem~\ref{thm:matfun_complex}, permute, and use the theorem again:
\begin{multline}
  \begin{bmatrix}
    \mathrm{Re} f
    \begin{pmatrix}
      A+ihE & B+ihF \\
      -(B+ihF) & A+ihE
    \end{pmatrix}        
    &
    \mathrm{Im} f
    \begin{pmatrix}
      A+ihE & B+ihF \\
      -(B+ihF) & A+ihE
    \end{pmatrix}    
    \\
    -\mathrm{Im} f
    \begin{pmatrix}
      A+ihE & B+ihF \\
      -(B+ihF) & A+ihE
    \end{pmatrix}    
    &
    \mathrm{Re} f
    \begin{pmatrix}
      A+ihE & B+ihF \\
      -(B+ihF) & A+ihE
    \end{pmatrix}    
  \end{bmatrix} 
  \\ =
  f
  \begin{pmatrix}
    A   & B   & hE  & hF \\
    -B  & A   & -hF & hE \\
    -hE & -hF & A   & B \\
    hF  & -hE & -B  & A
  \end{pmatrix} 
   =
  Pf
  \begin{pmatrix}
    A   & hE  & B   & hF \\
    -hE & A   & -hF & B \\
    -B  & -hF & A   & hE \\
    hF  & -B  & -hE & A
  \end{pmatrix}P  
  \\ =
  P
  \begin{bmatrix}
    \mathrm{Re} f
    \begin{pmatrix}
      A+iB & h(E+iF) \\
      -h(E+iF) & A+iB
    \end{pmatrix}        
    &
    \mathrm{Im} f
    \begin{pmatrix}
      A+iB & h(E+iF) \\
      -h(E+iF) & A+iB
    \end{pmatrix}    
    \\
    -\mathrm{Im} f
    \begin{pmatrix}
      A+iB & h(E+iF) \\
      -h(E+iF) & A+iB
    \end{pmatrix}    
    &
    \mathrm{Re} f
    \begin{pmatrix}
      A+iB & h(E+iF) \\
      -h(E+iF) & A+iB
    \end{pmatrix}    
  \end{bmatrix} 
  P
\end{multline}
where
\begin{align}
  P =
  \begin{bmatrix}
    I & 0 & 0 & 0 \\
    0 & 0 & I & 0 \\
    0 & I & 0 & 0 \\
    0 & 0 & 0 & I     
  \end{bmatrix}      
\end{align}
The real and imaginary parts of the sought matrix are in the $1,2$ and $1,4$ blocks of the
permuted matrix, respectively, and we therefore get the desired result
\begin{align}\label{eq:complex_step_for_complex_matrices}
  L_f(A+iB,E+iF) = \frac{1}{h}\left[f
    \begin{pmatrix}
      A+iB & h(E+iF) \\
      -h(E+iF) & A+iB    
    \end{pmatrix}
    \right]_{1,2}
  + O(h^2)
\end{align}
which can be seen as a complex step approximation for complex matrices.
We will below use \eqref{eq:complex_step_for_complex_matrices} to
approximate $\mathrm{Im}[L_f(A+ihB,E+ihF)]$ with real
$A,B,E,F$. In this case
we gain one order in the leading error term of
\eqref{eq:error_in_complex_step_for_complex_matrices}, because
\begin{align}\label{eq:error_term_special_case}
  \left[L_f^{(3)}\left(
    \begin{bmatrix}
      A & hB \\
      -hB & A 
    \end{bmatrix},
    \begin{bmatrix}
      E & hF \\
      hF & E
    \end{bmatrix}  
    \right)\right]_{1,2} = O(h).
\end{align}
This can be shown by employing Theorem~\ref{thm:exist_frechet} and applying $f(X_3)$ with
\begin{align}
  X_0 & =
  \begin{bmatrix}
    A & hB \\
    -hB & A 
  \end{bmatrix},\
  E_0 = \begin{bmatrix}
    E & hF \\
    hF & E 
  \end{bmatrix},
  \\ X_i & =
  \begin{bmatrix}
    X_{i-1} & E_{i-1} \\
    0      & X_{i-1}
  \end{bmatrix},\
  E_i =\begin{bmatrix}
  E_{i-1} & 0 \\
  0      & E_{i-1}
  \end{bmatrix},\ i = 1,2,3.
\end{align}
The term \eqref{eq:error_term_special_case} is given by the upper
right $[f(X_3)]_{1,16}$ block. A Taylor series representation of
$f(X_3)$, see \cite[Theorem
  4.7]{Higham_functions_of_matrices_2008}, gives a power series representation of $[f(X_3)]_{1,16}$ in
$h$ that can be used to show that $[f(X_3)]_{1,16} = O(h)$.

The approach above can be used also for higher derivatives. We will
derive the corresponding expression for the second order Fr{\'e}chet
derivative.  By Theorem~\ref{thm:matfun_complex} and the definition of
the second order Fr{\'e}chet derivative, we have that
\begin{multline}\label{eq:second_frechet_complex_block_form}
  L_f^{(2)}\left(
  \begin{bmatrix}
    \phantom{-}A \ B \\
    -B \ A
  \end{bmatrix},    
  \begin{bmatrix}
    \phantom{-}E \ F \\
    -F \ E
  \end{bmatrix},    
  \begin{bmatrix}
    \phantom{-}G \ H \\
    -H \ G
  \end{bmatrix}
  \right)
  \\
  =
  \begin{bmatrix}
    \mathrm{Re}[L_f^{(2)}(A+iB,E+iF,G+iH)]  & \mathrm{Im}[L_f^{(2)}(A+iB,E+iF,G+iH)] \\
    -\mathrm{Im}[L_f^{(2)}(A+iB,E+iF,G+iH)] & \mathrm{Re}[L_f^{(2)}(A+iB,E+iF,G+iH)]
  \end{bmatrix}
\end{multline}

Since the three matrices on the left hand side are all real, we can
use the complex step approximation for higher order by Al-Mohy and
Arslan~\cite{complex_step_higher_frechet_2021} and get that
\begin{align}
  L_f^{(2)} & \left(
  \begin{bmatrix}
    \phantom{-}A \ B \\
    -B \ A
  \end{bmatrix},    
  \begin{bmatrix}
    \phantom{-}E \ F \\
    -F \ E
  \end{bmatrix},    
  \begin{bmatrix}
    \phantom{-}G \ H \\
    -H \ G
  \end{bmatrix}
  \right)
  \\
  & =
  \frac{1}{h}\mathrm{Im} [L_f\left(
  \begin{bmatrix}
    A+ihG & B+ihH \\
    -B-ihH & A+ihG 
  \end{bmatrix},
  \begin{bmatrix}
    E & F \\
    -F & E
  \end{bmatrix}  
  \right)] + O(h^2)
  \\
  & =
  \frac{1}{h}\mathrm{Im} [
    \frac{1}{h}
    \left[f(X)\right]_{1:2,3:4} + O(h^3)    
  ] + O(h^2)
  \\
  & =  
  \frac{1}{h^2}\mathrm{Im}\left[f(X)\right]_{1:2,3:4} + O(h^2),
\end{align}
where
\begin{align}
  X = 
  \begin{bmatrix}
    A+ihG &  B+ihH &  hE    & hF    \\
    -B-ihH &  A+ihG & -hF    & hE    \\
    -hE    & -hF    &  A+ihG & B+ihH \\
    hF    & -hE    & -B-ihH & A+ihG
  \end{bmatrix},
\end{align}
and where we used~\eqref{eq:complex_step_for_complex_matrices} in the
second equality making use of~\eqref{eq:error_term_special_case} to
gain one order in the leading error term.
In combination
with~\eqref{eq:second_frechet_complex_block_form} this gives us a
complex step approximation for second order Fr{\'e}chet derivatives
with complex matrix input, i.e.\
$L_f^{(2)}(A+iB,E+iF,G+iH) = \frac{1}{h^2}(\mathrm{Im}\left[f(X)\right]_{13} + i\mathrm{Im}\left[f(X)\right]_{14}) + O(h^2)$.
Again, this approximation is not quite on the desired form. We use the
same approach as before, use Theorem~\ref{thm:matfun_complex},
permute, and use the theorem again:
\begin{align}
  \begin{bmatrix}
    \mathrm{Re} f(X)
    &
    \mathrm{Im} f(X)
    \\
    -\mathrm{Im} f(X)
    &
    \mathrm{Re} f(X)
  \end{bmatrix}
\end{align}
\begin{align}
  =&
  f
  \begin{pmatrix}
    A   & B   & hE  & hF  & hG  & hH  & 0   & 0  \\
    -B  & A   & -hF & hE  & -hH & hG  & 0   & 0  \\
    -hE & -hF & A   & B   & 0   & 0   & hG  & hH \\
    hF  & -hE & -B  & A   & 0   & 0   & -hH & hG \\
    -hG & -hH & 0   & 0   & A   & B   & hE  & hF \\
    hH  & -hG & 0   & 0   & -B  & A   & -hF & hE \\
    0   & 0   & -hG & -hH & -hE & -hF & A   & B  \\
    0   & 0   & hH  & -hG & hF  & -hE & -B  & A  
  \end{pmatrix}
  \\
  =&
  P
  f
  \begin{pmatrix}
    A   & hE  & hG  & 0   & B   & hF  & hH  & 0  \\
    -hE & A   & 0   & hG  & -hF & B   & 0   & hH \\
    -hG & 0   & A   & hE  & -hH & 0   & B   & hF \\
    0   & -hG & -hE & A   & 0   & -hH & -hF & B  \\
    -B  & -hF & -hH & 0   & A   & hE  & hG  & 0  \\
    hF  & -B  & 0   & -hH & -hE & A   & 0   & hG \\
    hH  & 0   & -B  & -hF & -hG & 0   & A   & hE \\
    0   & hH  & hF  & -B  & 0   & -hG & -hE & A  
  \end{pmatrix}
  P^T
  \\
  =&
  P
  \begin{bmatrix}
    \mathrm{Re} f(Y)  & \mathrm{Im} f(Y) \\
    -\mathrm{Im} f(Y) & \mathrm{Re} f(Y)
  \end{bmatrix}
  P^T,
\end{align}
where
\begin{align}
P = 
\begin{bmatrix}
  I & 0 & 0 & 0 & 0 & 0 & 0 & 0 \\
  0 & 0 & 0 & 0 & I & 0 & 0 & 0 \\
  0 & I & 0 & 0 & 0 & 0 & 0 & 0 \\
  0 & 0 & 0 & 0 & 0 & I & 0 & 0 \\
  0 & 0 & I & 0 & 0 & 0 & 0 & 0 \\
  0 & 0 & 0 & 0 & 0 & 0 & I & 0 \\
  0 & 0 & 0 & I & 0 & 0 & 0 & 0 \\
  0 & 0 & 0 & 0 & 0 & 0 & 0 & I 
\end{bmatrix}
\end{align}
and
\begin{align}\label{eq:X_tmp}
Y = \begin{bmatrix}
  A+iB    & h(E+iF)  & h(G+iH)  & 0       \\
  -h(E+iF) & A+iB    & 0        & h(G+iH) \\
  -h(G+iH) & 0       & A+iB     & h(E+iF) \\
  0       & -h(G+iH) & -h(E+iF) & A+iB    
\end{bmatrix}.
\end{align}
The real and imaginary parts of the sought matrix are in the $1,4$ and
$1,8$ blocks of the permuted matrix, respectively, and we therefore
get the desired result
\begin{align}
  L_f^{(2)}(A+iB, E+iF, G+iH) = \frac{1}{h^2} \left[f(Y)\right]_{1,4} + O(h^2).
\end{align}

Finally, we develop a complex step approximation for a second order
partial derivative on the form $\partial^\alpha
f(\bar{A}(x))|_{x=\bar{x}}$, $|\alpha|=2$. Let $\alpha=\beta+\gamma$
with $|\beta|=|\gamma|=1$. Then, according to Theorem~\ref{thm:derivative_as_linear_combination_of_frechet},
\begin{align}\label{eq:partial_second_order_as_sum_of_frechet}
  \partial^\alpha f(\bar{A}(x))|_{x=\bar{x}} =
  L_f(\bar{A},\bar{A}^{(\alpha)}) +
  L_f^{(2)}(\bar{A},\bar{A}^{(\beta)},\bar{A}^{(\gamma)})
\end{align}
and we can use the complex step approximations developed above for the
first and second order Fr{\'e}chet derivatives. Another alternative is
to compute the first and second order contributions simultaneously.
We decompose $\bar{A}(x)$ and its derivatives at $x=\bar{x}$ in their
real and imaginary parts: $\bar{A} = A+iB$, $\bar{A}^{(\beta)} =
E+iF$, $\bar{A}^{(\gamma)} = G+iH$, $\bar{A}^{(\alpha)} = K+iL$ and let
\begin{align}
  X = 
  \begin{bmatrix}
    A+ihG  &    B+ihH  &  h(E+ihK) & h(F+ihL) \\
  -(B+ihH) &    A+ihG  & -h(F+ihL) & h(E+ihK) \\
 -h(E+ihK) & -h(F+ihL) &    A+ihG  &   B+ihH  \\
  h(F+ihL) & -h(E+ihK) &  -(B+ihH) &   A+ihG
  \end{bmatrix}.
\end{align}
Then,
\begin{align}
  & \frac{1}{h^2}  \textrm{Im} [f(X)]_{1:2,3:4} \\
  = & \frac{1}{h}  \textrm{Im}[L_f\left(
    \begin{bmatrix}
      A+ihG  & B+ihH \\
    -(B+ihH) & A+ihG      
    \end{bmatrix},
    \begin{bmatrix}
      E+ihK  & F+ihL \\
    -(F+ihL) & E+ihK
    \end{bmatrix}    
    \right)] + O(h^2) \nonumber
  \\
  = & \frac{1}{h}  \textrm{Im}[L_f\left(
    \begin{bmatrix}
      A+ihG  & B+ihH \\
    -(B+ihH) & A+ihG      
    \end{bmatrix},
    \begin{bmatrix}
      E & F \\
     -F & E
    \end{bmatrix}    
    \right)] \nonumber
  \\
  & + \textrm{Re}[L_f\left(
    \begin{bmatrix}
      A+ihG  & B+ihH \\
    -(B+ihH) & A+ihG      
    \end{bmatrix},
    \begin{bmatrix}
      K & L \\
     -L & K
    \end{bmatrix}    
    \right)] + O(h^2) \nonumber
  \\
  = &
  L_f^{(2)} \left(
  \begin{bmatrix}
    \phantom{-}A \ B \\
    -B \ A
  \end{bmatrix},    
  \begin{bmatrix}
    \phantom{-}E \ F \\
    -F \ E
  \end{bmatrix},    
  \begin{bmatrix}
    \phantom{-}G \ H \\
    -H \ G
  \end{bmatrix}
  \right) 
   + L_f\left(
  \begin{bmatrix}
    \phantom{-}A \ B \\
    -B \ A
  \end{bmatrix},
  \begin{bmatrix}
    \phantom{-}K \ L \\
    -L \ K
  \end{bmatrix}    
  \right)] + O(h^2) \nonumber
    \\
    = &   
    \begin{bmatrix}
      \textrm{Re}[\partial^\alpha f(\bar{A}(x))|_{x=\bar{x}}] & \textrm{Im}[\partial^\alpha f(\bar{A}(x))|_{x=\bar{x}}] \\
      -\textrm{Im}[\partial^\alpha f(\bar{A}(x))|_{x=\bar{x}}] & \textrm{Re}[\partial^\alpha f(\bar{A}(x))|_{x=\bar{x}}]
    \end{bmatrix} + O(h^2). \nonumber
\end{align}
Here, we first used \eqref{eq:complex_step_for_complex_matrices} again
making use of~\eqref{eq:error_term_special_case} to see that this is a
second order approximation. The second equality is due to the
linearity of the Fr{\'e}chet derivative and the third step is due
to~\cite[Theorem~3.1]{complex_step_higher_frechet_2021}. The last step
is due to \eqref{eq:frechet_complex_block_form},
\eqref{eq:second_frechet_complex_block_form}, and
\eqref{eq:partial_second_order_as_sum_of_frechet}. Thus, we have that  
\begin{equation}
  \partial^\alpha f(\bar{A}(x))|_{x=\bar{x}} = \frac{1}{h^2} (\textrm{Im}[f(X)]_{13} + i \textrm{Im}[f(X)]_{14}) + O(h^2).
\end{equation}
Using the permutation approach again, we arrive at the desired result: 
\begin{align}
\partial^\alpha f(\bar{A}(x))|_{x=\bar{x}} = \frac{1}{h^2}
\left[f
\begin{pmatrix}
  \bar{A} & h\bar{A}^{(\beta)} & h\bar{A}^{(\gamma)} & h^2\bar{A}^{(\alpha)} \\
  -h\bar{A}^{(\beta)} & \bar{A} & -h^2\bar{A}^{(\alpha)} & h\bar{A}^{(\gamma)} \\
  -h\bar{A}^{(\gamma)} & -h^2\bar{A}^{(\alpha)} & \bar{A} & h\bar{A}^{(\beta)} \\
  h^2\bar{A}^{(\alpha)} & -h\bar{A}^{(\gamma)} & -h\bar{A}^{(\beta)} & \bar{A}
\end{pmatrix}
\right]_{14}
+ O(h^2).
\end{align}

\bibliographystyle{siamplain}
\bibliography{biblio}

\end{document}